\documentclass[reqno, 11pt]{amsart}
 \usepackage{color,mathrsfs}
\usepackage[utf8]{inputenc}
 \usepackage{fullpage}  \usepackage{srcltx} \usepackage{pdfsync}
\usepackage[colorlinks=true, pdfstartview=FitV, linkcolor=blue, citecolor=blue, urlcolor=blue,pagebackref=false,breaklinks=true]{hyperref}
\usepackage{calc,amsfonts,amsthm,amscd,epsfig,psfrag,amsmath,amssymb,enumerate,graphicx}
\usepackage[showlabels,sections,floats,textmath,displaymath]{}

\reversemarginpar
\newlength\fullwidth
\numberwithin{equation}{section}

\DeclareMathSymbol{\leqslant}{\mathalpha}{AMSa}{"36} 
\DeclareMathSymbol{\geqslant}{\mathalpha}{AMSa}{"3E} 
\DeclareMathSymbol{\eset}{\mathalpha}{AMSb}{"3F}     
\newcommand{\maxtwo}[2]{\max_{\substack{#1 \\ #2}}} 
 \def\1{\ifmmode {1\hskip -3pt
    \rm{I}} \else {\hbox {$1\hskip -3pt \rm{I}$}}\fi}

\newcommand{\sumtwo}[2]{\sum_{\substack{#1 \\ #2}}} 
 
\newcommand{\bff}{{\bf f}} 

\newcommand{\Z}{\mathbb{Z}}

\newtheorem{theorem}{Theorem}[section] 
 
\newtheorem{lemma}[theorem]{Lemma} 
\newtheorem{proposition}[theorem]{Proposition} 
\newtheorem{rem}[theorem]{Remark} 
 
\newtheorem{remark}[theorem]{Remark}

\newcommand{\N}{\mathbb N}
\newcommand{\bP}{{\bf P}}

\newcommand{\bbE}{{\ensuremath{\mathbb E}} }

\newcommand{\bbN}{{\ensuremath{\mathbb N}} } 
 
\newcommand{\bbP}{{\ensuremath{\mathbb P}} } 
 
\newcommand{\bbR}{{\ensuremath{\mathbb R}} }

\newcommand{\bbZ}{{\ensuremath{\mathbb Z}} }

\newcommand{\gep}{\varepsilon}

\newcommand{\gO}{\Omega}

\newcommand{\gd}{\delta}
\newcommand{\dd}{\mathrm{d}}

\newcommand{\gl}{\lambda}

\newcommand{\DD}{\mathcal{D}}

\newcommand{\gb}{\beta}
\newcommand{\gL}{\Lambda}
\newcommand{\gD}{\Delta}

\newcommand{\be}{\mathbf{e}}

\newcommand{\mal}{\mathcal A_L}

\author[H. Lacoin]{Hubert Lacoin}
\address{H. Lacoin, 
CEREMADE - UMR CNRS 7534 - Universit\'e Paris Dauphine,
Place du Mar\'echal de Lattre de Tassigny, 75775 CEDEX-16 Paris, France. \newline
e--mail: {\tt lacoin@ceremade.dauphine.fr}}

\begin{document}
 
\title[Scaling limit for 2D Ising droplets]{The scaling limit for zero temperature planar Ising droplets: with and without magnetic fields}

\begin{abstract}
  We consider the continuous time, zero-temperature heat-bath
  dynamics for the nearest-neighbor Ising model on $\mathbb Z^2$ with positive magnetic field.
  For a system of size $L\in \bbN$, we start with initial condition $\sigma$ such that $\sigma_x=-1$ if $x\in[-L,L]^2$ and
  $\sigma_x=+1$ and investigate the scaling limit of the set of $-$ spins when both time and space are rescaled by $L$.
  We compare the obtained result and its proof with the case of zero-magnetic fields, for which a scaling result was proved in \cite{cf:LST}. In that case,
  the time-scaling is diffusive and
  the scaling limit is given by anisotropic motion by curvature.
\end{abstract}

\maketitle

\section{Introduction}

The Ising model is one of the simplest model proposed by statistical mechanics to investigate ferromagnetic properties of metals.
It is based on the following simplification of reality: we assume that a ferromagnetic material is composed of macroscopic magnet 
that live on a lattice and can assume only two orientation, up (or $+$) and down (or $-$) called spins; the way orientation of the micromagnet are 
determined follows two rules: neighboring magnets like to have the same orientation, and all spins like to align with the external magnetic 
fields if there is one.

\medskip

The Ising model is a probabilistic model (i.e.\ a probability law on the set of possible spin orientation) following these rules and is defined 
using a Boltzman-Gibbs formalism, which is detailed in the next section. It describes the equilibrium state of a ferromagnet. 

\medskip

The Stochastic Ising model also called Glauber dynamics or heat-bath dynamics for Ising model, is a Markov chain on the set of spin configuration that 
describes the evolution of a magnet out of equilibrium 
(e.g.\ after a brutal change of temperature or of external magnetic fields). This is the object we study in this paper. Our aim is to understand
the time needed and the pattern used for a magnetic material to reach its equilibrium state after a change of external conditions.
In order to bring a more complete answer to these questions, we consider them in a simplified but still non-trivial setup that is the zero-temperature limit.
This brings heuristic intuition for what should happen in the whole low temperature regime.

\subsection{The Glauber Dynamics for the Ising Model}

Set $$\bbZ^*:=\bbZ+\frac{1}{2}:=\{x+(1/2)\ | \ x\in \bbZ\}.$$
Consider the square $$\gL=\gL_L:=[-L,L]^2\cap (\bbZ^*)^2.$$
We define the set of spin configuration on $\gL_L$ to be $\{-1,1\}^{\gL_L}$. A generic spin configuration is denoted
 by $\sigma=(\sigma_x)_{x\in\gL_L}$ and 
$\sigma_x\in \{-1,1\}$ is called the spin at site $x$.
Define the external boundary of $\gL$ as 
\begin{equation}
 \partial \gL:=\{y\in \bbZ^d \setminus \gL \ | \ \exists x \in \gL, x \sim y\}.
\end{equation}

The Ising model at inverse temperature $\gb$, with external magnetic fields $h$ and boundary condition $\eta \in \{-1,1\}^{\partial \gL_L}$
is a measure on the set of spin configurations defined by

\begin{equation}
 \mu^{\gb,h,\eta}_L(\sigma):=\frac{1}{Z^{h.\gb}_{\gL}}\exp\left(\gb\sumtwo{\{x,y\}\subset (\gL \cup \partial \gL)}{x\sim y} \sigma_x\sigma_y+h
\sum_{x\in \gL} \sigma_x\right),
\end{equation}
 where
\begin{equation}
 Z^{h,\gb}_{\gL}:=\sum_{\sigma \in \{-1,1\}^\gL} \exp\left(\gb\sumtwo{\{x,y\}\subset (\gL \cup \partial \gL)}{x\sim y} \sigma_x\sigma_y+h
\sum_{x\in \gL} \sigma_x\right),
\end{equation}
and the convention is taken that $\sigma_x:=\eta_x$ when $x\notin \gL$.
Note that the first term in the exponential makes the spins of neighboring sites more likely to agree whereas the second term underlines that configuration with 
spins aligned with 
the magnetic fields are favored.

\medskip

The heat-bath dynamic for the Ising model (at inverse temperature $\gb$, with external magnetic fields $h$ and boundary condition $\eta$)
is a Markov chain on the set of spin configuration $\{-1,1\}^\gL$. We denote the trajectory of the Markov chain by $\sigma(t)=(\sigma_x(t))_{x\in \gL}$.
One starts from a given configuration $\sigma_0$
and the rules for the evolution are the following:
\begin{itemize}
\item[(i)] Sites $x\in \gL$ are equipped with independent rate one Poisson process: 
$(\tau^x_{n})_{n\ge 0}$ where $\tau^x_{0}=0$, i.e.\ the increments $(\tau_{n,x}-\tau_{n-1,x})_{n\ge 0, x\in \bbZ^d}$ are IID exponential 
variables. 
\item[(ii)] The spin at site $x$ may change its 
value only when the clock at $x$ rings, i.e.\ at time $\tau^x_{n}$, $n\ge 1$, and its value stays constant on the
intervals $[\tau^x_n,\tau^x_{n+1})$, $n\ge 0$.
\item[(iii)] When a clock rings, the spin at $x$ is updated according to
the following law independently of the past history of the process:
\begin{itemize}
 \item $\sigma_x(t_0)=+$ with probability 
 \begin{equation}\label{turnplus}
        \frac{\exp\left(\gb\sum_{y\sim x}\sigma_y(\tau^-_{n,x})+h\right)}{2\cosh(\gb\sum_{y\sim x}\sigma_y(\tau^-_{n,x})+h)},                                   
                                          \end{equation}

 \item $\sigma_x(t_0)=-$ with probability 
 \begin{equation}\label{turnminus}
          \frac{\exp\left(-\gb\sum_{y\sim x}\sigma_y(\tau^-_{n,x})+h\right)}{2\cosh(\gb\sum_{y\sim x}\sigma_y(\tau^-_{n,x})+h)},
                                          \end{equation}

\end{itemize}
\end{itemize}

When $x$ has a neighbor $y$ in $\partial \gL$, the value $\sigma_y$ appearing in the equations \eqref{turnplus}-\eqref{turnminus}
is fixed (by convention) to be equal to $\eta_y$ for all time, where $\eta$ is the boundary condition. 
The update of the spin corresponds to sampling a spin configuration according to the measure 
$\mu^{\gb,h,\eta}_\gL(\cdot \ | \ \sigma_y= \sigma_y(\tau^-_{n,x}), \forall y\ne x)$.

\medskip

As a consequence to this last remark, $\mu^{\gb,h,\eta}_\gL$ is the unique invariant measure for the dynamics,
so that the law of $\sigma(t)$ converges to the equilibrium measure $\mu^{\gb,h,\eta}_\gL$.

\medskip

The main questions in the study of dynamics is how much time is needed to reach equilibrium 
and what is the pattern used by the chain to reach it when 
we consider the dynamics on a very large domain $\gL$ .
The answer to this question should of course depend on the temperature $\gb$, the magnetic fields $h$, the boundary condition $\eta$ 
the and the initial condition $\sigma_0$.

\medskip

In what follows, we denote often by $+$ resp.\ $-$ the spin configuration or boundary condition where all spins are $+$ resp.\ $-$.

%

\subsection{Conjecture and known results in the case $\gb\in(0,\infty)$, $h=0$}

Let us review shortly what is known and conjectured for these kind of dynamics when $\gb\in(0,\infty)$ and $h=0$.
The property of the dynamics depends crucially on the equilibrium property of the system.
Recall that the two dimensional Ising model undertakes a phase transition at $\gb_c=\log(1+\sqrt{2})/2$ (see the seminal work of Onsager \cite{cf:Onz})
which has the following form:
\begin{itemize}
 \item When $\gb<\gb_c$, correlation between spin at different sites decay exponentially, and for this reason 
what happens in the center of $\gL_L$ becomes independent of the
boundary condition when $L$ tends to infinity. This is called the high temperature phase. 
 \item When $\gb>\gb_c$, on the contrary, long-range correlation are present between spins, and boundary condition plays a crucial role.
In particular the measure $\mu^{\gb,h,+}_L$ and $\mu^{\gb,h,-}_L$  corresponding to $+$ and $-$ boundary condition are very different.
This is called the low temperature phase.
\end{itemize}

\medskip

In the high temperature phase, the rapid decay of correlation between distant sites makes the evolution of the system in two 
distant zones of the box $\gL_L$ almost independent. Schematically, the box of $\gL_L$ can be separated in $O(L^2)$ zones of 
finite size that come to equilibrium independently. This requires a time of order $\log L$.
This prediction has been made rigorous by Lubetzky and Sly, who proved that in that case, the mixing time (i.e.\ the time to reach equilibrium) 
for the dynamic in $\gL_L$
is equal to $\gl_{\infty}\log L(1+o(1))$ where $\gl_{\infty}$ is the relaxation time for the infinite volume dynamics (see \cite{cf:LS}).
For a formal definition of mixing time and relaxation time and an introduction to the modern theory of Markov chain, we refer to \cite{LPW}.

\medskip

In the low temperature phase the behavior of the dynamics depends on the boundary condition and for the sake of simplicity we restrict 
to the case of $+$ boundary condition. In that case, the equilibrium state is biased towards plus, and even spins in the center
have a larger probability to be $+$ than minus. In a sense, one can say that, at equilibrium, the center of the box ``knows'' what the boundary condition is. 
Thus, if one starts
e.g.\ from full $-$ initial condition ($\sigma_x(0)=-1$ for all $x\in \gL$), information must travel from the boundary to the center of the box in order to 
reach equilibrium.
For this reason the mixing time is much longer than in the high temperature phase.

\medskip

In \cite{cf:Lifshitz}, Lifshitz described a conjectural pattern used by the system with $+$ boundary condition to reach equilibrium that can be described as follows:
starting from $-$ initial condition the system should rapidly reach a state of local equilibrium that looks like the equilibrium measure
with $-$ boundary condition, (we call this the $-$ phase); then on the time-scale $L^2$, something looking like the true equilibrium measure with $+$ boundary condition, the $+$ phase,
should start to appear in the neighborhood of the cubes boundary. The interface between the $+$ and the $-$ phase should 
move on the diffusive time-scale $L^2$, having a drift in time proportional to its local curvature.
As a consequence the system should reach equilibrium when 
the bubble formed by the $-$ phase disappears macroscopically, \textit{i.e.} in a time $O(L^2)$.

\medskip

For finite $\gb>\gb_c$ this conjecture is far from being on rigorous mathematical ground, but Lifshitz ideas have been used to get bounds 
on the mixing time. The best to date being by Lubetzky \textit{et al.} \cite{cf:LSMT} saying that the system reaches equilibrium in a time
$L^{\log L}$, still far from the conjecture $L^2$. This gap between the Lifshitz conjecture and the rigorous mathematical result has been 
one of the incentives to study the simpler 
zero-temperature version of the model.

\medskip

Dynamics with different boundary condition or non-zero magnetic fields at low temperature also exhibits interesting behavior 
like low-temperature induced metastability  (see e.g.\ \cite{cf:Scho}) that we choose not to expose here. 
Note also that results are available at the critical temperature $\gb_c$ 
\cite{cf:SL2}, where the equilibrium state of the system is somehow harder to describe.

\section{Zero-temperature dynamics}

Given $h\ge 0$, we look at the limiting dynamics when $\gb$ tends to infinity.
We call this the zero-temperature limit (recall that $\gb$ is the inverse of the temperature).
In what follows we will consider only $+$ boundary condition: $\eta_x=+1, \ \forall x \in \partial \gL$.

\medskip

With this setup, on a finite box with $+$ boundary condition, the limit of the Ising measure $\lim_{\gb\to \infty}\mu_L^{\gb,h,+}$ 
is just the Dirac measure on the full $+$ configuration: $\sigma_x=+1,\ \forall x\in \gL_L$.
The limiting dynamics is non-trivial and can be described as follows:
the value the spin at $x$ is updated with rate one as before; when a spin is updated
it takes the value of the majority of its neighbors if it is well defined
and take value $\pm 1$ with probability $e^{\pm h}/(2\cosh(h))$ if it has the same number of $+$ and $-$ neighbors.
The process can also be defined when $h=\infty$, it corresponds to the case where both $\gb$ and $h$ 
tend to infinity with $h\ll \gb$.

\medskip

The question concerning the pattern used to reach the equilibrium then takes the following form: 
starting from a finite domain of $\bbZ^2$ filled with $-$ spins what is the time needed to
reach the whole $+$ configuration and what is the pattern used to reach it. 

\medskip
More precisely:
We consider a compact, simply connected subset $\mathcal D\subset
[-1,1]^2$ whose boundary is a closed smooth curve.  Given $L\in \N$ we
consider the Markov chain described above with  initial 
condition
\begin{equation}\label{bien}
\sigma_x(0)=\begin{cases} - 1 \quad &\text{ if } x\in (\bbZ^*)^2\cap L\mathcal D,
                          \\ +1 \quad &\text{ otherwise}.
            \end{cases}
 \end{equation}
In order to see a set of "$-$" spins as a subset of $\bbR^2$, each
vertex $x\in (\bbZ^*)^2$ may be identified with the closed square of
side-length one centered at $x$,
\begin{equation}
\mathcal C_x:= x+[-1/2,1/2]^2.
\end{equation}
One defines 
\begin{equation}\label{mal}
\mal(t):=\bigcup_{\{x: \ \sigma_x(t)=-1\}} \mathcal C_x,
\end{equation}
which is the ``$-$ droplet'' at time $t$ for the dynamics. Note that the
boundary of $\mal(t)$ is a union of edges of $\Z^2$ (and this is the
reason why we defined the Ising model on  $(\Z^*)^2$ rather than on $\Z^2$).

\medskip

We investigate the scaling limit of $\mal(t)$ and the time needed for it to vanish. 
The nature of the scaling limit depends really much on the value of the external magnetic fields $h$:
\begin{itemize}
 \item If $h>0$, the interface between $+$ and $-$ is always pushed towards the minus region, at linear speed, and
thus macroscopic motion is visible on time-scale $L$.  
 \item If $h=0$, there is no sign that is favored, and the interface will be pushed in the direction of its curvature to
reduce its length. This is visible only on the diffusive time scale.
\end{itemize}

The main aims of this paper are: to review in detail the recent proof (with co-authors \cite{cf:LST}) that the scaling 
limit of  $\mal(t)$ is given by an anisotropic curve shortening flow, and to give a description of the scaling limit 
in the case $h>0$ (for the sake of simplicity we will limit our proof to $h=\infty$).

\subsection{The case of zero magnetic fields}
                            
Let us focus first on the case of zero-magnetic fields, and make more precise the Lifshitz conjecture \cite{cf:Lifshitz} 
concerning low temperature dynamics in this case.

On heuristic
grounds, Lifshitz predicted that the boundary of $\mal(t)$ should
follow an anisotropic curve shortening  motion: after rescaling space
by $L$ and time by $L^2$ and letting $L$ tend to infinity, the motion
of the interface between $\mal(t)$ and its complement (\textit{i.e.} between $+$ and $-$ spins) should be
deterministic and the local drift of the interface should be
proportional to the curvature. An anisotropic factor should appear in front of the curvature to
reflect anisotropy of the lattice.  Spohn \cite{cf:Spohn} made this conjecture more precise and brought some elements for its proof: let
$\gamma(t,L)$ denote the boundary of the (random) set $(1/L)\mal(L^2
t)$. Then, for $L\to\infty$, the flow of curve $(\gamma(t,L))_{t\ge 0}$ should converge to a
deterministic flow $(\gamma(t))_{t\ge0}$ and the motion of the limiting curve
 should be such that the normal velocity at a
point $x\in\gamma(t)$ is given by the curvature at $x$ multiplied an anisotropic
factor $a(\theta_x)$, where $\theta_x$ is the angle made by the outward
directed normal to $\gamma(t)$ at $x$ together the horizontal axis
 (see Figure \ref{fig:suppfunc}). The velocity points in the direction of concavity.
The function $a(\cdot)$ has the explicit expression
\begin{equation}
  \label{eq:a}
  a(\theta):=\frac1{2(|\cos(\theta)|+|\sin(\theta)|)^2}.
\end{equation}
The area enclosed by the curve decreases with a constant speed:  $-\int_{0}^{2\pi} a(\theta)\dd \theta$.
In particular, the curve $\gamma(t)$ shrinks to a point in a finite
time 
\[
t_0=\frac{Area(\mathcal D)}{\int_0^{2\pi} a(\theta) \dd\theta}=
\frac{Area(\mathcal D)}2.
\]

Note that the function $a(\cdot)$ is symmetric around $0$ and is periodic
with period $\pi/2$. This is inherited from the discrete symmetries of the lattice
$(\bbZ^*)^2$.

 \begin{figure}[hlt]
\leavevmode
\epsfysize =5.8 cm
\psfragscanon 
\psfrag{theta}{\small $\theta$}
\psfrag{O}{\small $\bf{0}$}
\psfrag{x}{\small $x(\theta)$}
\psfrag{N}{}
\psfrag{k}{\small $k(\theta)$}
\psfrag{V}{\small $v(\theta)$}
\psfrag{h}{\small $h(\theta)$}
\epsfbox{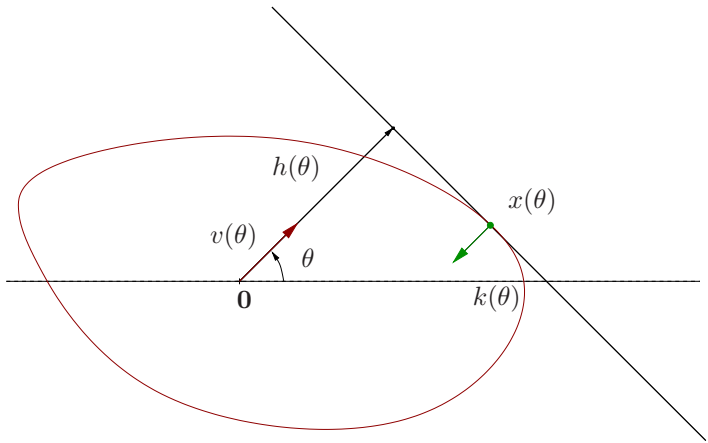}
\begin{center}
\caption{
 A graphic description of the support function $h$.
Given $\theta$, consider the point $x(\theta)$ of $\gamma$ that maximizes 
$x\cdot v(\theta)$ (it is unique if the curve is strictly convex).
Then $h(\theta)=x(\theta)\cdot v(\theta)$, and $k(\theta)$ is the norm of the curvature vector of $\gamma$ (green vector)
 at  $x(\theta)$. If the tangent to $\gamma$ at $x$ exists it is normal to $v(\theta)$ and $|h(\theta)|$ 
is the distance between the tangent and the origin.}
\label{fig:suppfunc}
\end{center}
\end{figure}

In a recent work \cite{cf:LST}, with co-authors, this conjecture was brought on rigorous ground in the case when the initial condition 
$\mathcal D$ is convex.
Let us first give a rigorous definition of motion by curvature and a result concerning its existence:
given a convex smooth closed curve $\gamma=\partial \DD$ in $\bbR^2$, we parameterize it
following a standard convention of convex geometry
(cf. e.g. \cite{GL2} and Figure \ref{fig:suppfunc}). 
For $\theta\in[0,2\pi]$ let ${ v}(\theta)$ be the unit vector forming an anticlockwise angle
$\theta$ with the horizontal axis
and let
\begin{equation}\label{support}
h(\theta)=\sup\{x\cdot v(\theta), x\in\gamma\},
\end{equation}
where $\cdot$ denotes the usual scalar product in $\bbR^2$.
The function $\theta\mapsto h(\theta)$ (called ``the support function'') 
uniquely determines $\gamma$:
\begin{equation}
  \DD=\bigcap_{0\le \theta\le 2\pi}\{x\in\bbR^2:x\cdot v(\theta)\le h(\theta)\}.
\end{equation}
With this parameterization, the anisotropic curve shortening evolution reads
\begin{equation}
  \label{eq:meanc}
  \begin{cases}
\partial_t h(\theta,t)=- a(\theta)k(\theta,t),      \\
h(\theta,0)=h(\theta),
  \end{cases}
\end{equation}
where, for a convex curve $\gamma$,
$k(\theta)\ge0$ is the curvature at the point $x(\theta)\in\gamma$ 
where the outward normal vector makes  an angle
$\theta$ with the horizontal axis and
the time derivative is taken at constant $\theta$ (see \cite[Lemma 2.1]{GL2} for a proof that \eqref{eq:meanc}
is equivalent to the standard definition of anisotropic motion by curvature). 

\medskip

Existence of a solution to \label{eq:mean} is not straight-forward. It was proved under the assumption that $a$ is $C^2$ in \cite{GL2}.
The function $a$ given by equation \eqref{eq:a} is only Lifshitz due to singularity at $\theta=i\pi/2$, $i=1,\dots,4$ but a proof of existence and uniqueness of 
the motion was given in \cite{cf:LST} for that particular case.

\begin{theorem}[{\cite[Theorem 2.1]{cf:LST}}]
\label{th:deterministico}

Let $\mathcal D\subset[-1,1]^2$ be strictly convex and assume that its boundary
$\gamma=\partial \mathcal D$ is a curve whose curvature $[0,2\pi]\ni
\theta\mapsto k(\theta)$ defines a positive, $2\pi$-periodic,
Lipschitz function.

\medskip

 Then there exists a unique flow of convex curves
$(\gamma(t))_t$ with curvature defined everywhere, such that $\gamma(0)=\gamma$ and that the corresponding
support function $h(\theta,t)$ solves \eqref{eq:meanc} for $t\ge0$ and
satisfies the correct initial condition $h(\theta,0)=h(\theta)$.

\medskip

The
curve $\gamma(t)$ shrinks to a point ${\bf x}_f\in\bbR^2$ at time
$t_f=Area(\mathcal D)/2$.

\medskip

For $t<t_f$, $\gamma(t)$ is a smooth curve
in the following sense: its curvature function $k(\cdot,t)$ is
Lipschitz and bounded away from $0$ and infinity on any compact subset
of $[0,t_f)$.
\end{theorem}

Now let $\mathcal D(t)$ denotes the flow of convex shape whose support function is solution of \eqref{eq:meanc}, with $\mathcal D(0)=\mathcal D$.
Let $B(x,r)$ denote the closed Euclidean ball of radius $r>0$ and center $x\in \bbR^2$. For a closed set $\mathcal C\subset \bbR^2$, define the inner
and outer $\delta$-neighborhood of $\mathcal C$ to be
\begin{equation}\label{lesvoisins}\begin{split}
 \mathcal C^{(\delta)}&:=\bigcup_{x\in \mathcal C}B(x,\delta),\\
 \mathcal C^{(-\delta)}&:=\left(\bigcup_{x\in \mathcal C^c}B(x,\delta)\right)^c.
\end{split}
\end{equation}

We say that an event, or rather, that a sequence of events $(A_n)_{n\ge 0}$ holds {\bf with high probability} of w.h.p.\ if the probability of $A_n$ tends to one.
Then $\mathcal D(t)$ is the scaling limit $\mal(t)$ in the following sense.

\begin{theorem}[{\cite[Theorem 2.2]{cf:LST}}]
\label{th:convex}
Consider the dynamics starting with initial condition given by \eqref{bien}, where $\mathcal D$ is a convex shape satisfying the assumption of 
Theorem \ref{th:deterministico}.
For any $\delta>0$ one has w.h.p.
\begin{gather}\label{eq:scaling}
  \mathcal D^{(-\delta)}(t)\subset \frac1L \mal( L^2 t)\subset  \mathcal D^{(\delta)}(t)
\quad\quad \text{for every}\quad 0\le t\le t_f+\delta\\
\mal(L^2 t)=\emptyset \quad\quad \text{for every}\quad t>t_f+\delta.
\end{gather}
In particular,
 one has the following convergence in probability:
\begin{equation}
\label{eq:drif}
  \lim_{L\to\infty} \frac{\tau_+}{L^2 Area( \mathcal D)}=\frac{1}{2}.
\end{equation}
\end{theorem}

We can mention previous related result in the literature: in \cite{cf:CSS}, the authors consider a simplified dynamics
that do not allow the interface to break in several components, for this dynamics, they present a result  similar to 
\eqref{eq:drif} without any statement concerning the limiting shape.

\medskip

In \cite{cf:CL} the drift for the interface at the initial time is studied, and the author prove that it is proportional to the curvature
multiplied by an anisotropy function that is different from $a(\theta)$ of equation \eqref{eq:a}.
This difference is explained by the fact that the initial condition considered by the authors \eqref{bien} is very far from being a 
local equilibrium for the interface dynamics.

\begin{remark}\rm
Our result concerns only the case where $\mathcal D$ is a convex shape and there are serious reason for it.
The first one is that there is no proof of the existence of anisotropic curvature motion starting from non-convex initial condition
(see \cite{Grayson} in the isotropic case), and its regularity.  
On the probabilistic side, our method seem to be able to handle the non-convex case, but would need considerable improvements. 
\end{remark}

We can now compare this result with what happens in the case of positive magnetic fields.
%
%
%
%

\subsection{Zero-temperature dynamics with positive magnetic fields ($h=\infty$)}

Consider the dynamics where sites are still updated with rate $1$, but with the following rule of update: when  a site is updated,
its  spin flips to $+$ if it has two or more $+$ neighbors and to $-$ if has three or more $-$ neighbors.

\medskip

In that case, the right time-scale to describe the evolution of $\mal$ is not $L^2$ but $L$ and the interface is always 
going towards the $-$ side regardless of the curvature. The main new result of this paper is identifying the scaling limit in this case.

\medskip

Heuristically, the intensity of the drift of the interface can be deduced from mathematical work on Totally Asymmetric Simple Exclusion Process (TASEP),
more precisely from results giving the scaling limit of the height function (see Section \ref{prtsys} for detailed explanations). 
The drift at a point where the interface makes an angle $\theta$ with the horizontal axis is equal to
\begin{equation}\label{beteta}
 b(\theta):= \frac{|\sin(2\theta)|(|\cos\theta|+|\sin \theta |)}{1+|\sin 2\theta|}.
\end{equation}

\medskip

Let us give a rigorous definition of this shape evolution in the case of convex initial condition.
The shape remains convex at all times and one can describe the evolution of the interface in terms of the support function (recall \eqref{support}) as follows:
\begin{equation}\label{pmc}
\begin{cases} \partial_t h(\theta,t)=- b(\theta),\\
 h(\theta,0)=h(\theta).
\end{cases}
\end{equation}
The problem is that the equation \eqref{pmc} is not well-posed. 
However there is a notion of weak solution to \eqref{pmc} that has a rather simple description and
this is the one that will be of interest to us: given an initial condition $\mathcal D$, define
\begin{equation}\label{mosdef}
\mathcal D(t):=\bigcap_{\theta\in[0,2\pi]}\{x\in\bbR^2:x\cdot v(\theta)\le h(\theta)-b(\theta)t\}.
\end{equation}
It could be shown that the support function of $(\mathcal D(t))_{t\ge 0}$ is the unique viscosity solution of \eqref{pmc} (see \cite{cf:DL} for an introduction 
to this concept) but we will 
not pain ourselves with such considerations, and rather consider \eqref{mosdef} as a definition.

\medskip
Our main result is

\begin{theorem}\label{resap}
The process
$(\frac{1}{L}\mal(Lt))_{t\ge 0}$ converges to $(\mathcal D(t))_{t\ge 0}$ in probability, 
in the topology of time-uniform convergence for the Hausdorf metric. This is to say,  for any $\delta>0$
with probability tending to one,
\begin{equation}
\mathcal D^{(-\delta)}(t)\subset (\frac{1}{L}\mal(Lt))_{t\ge 0}\subset  \mathcal D^{(\delta)}(t).
\end{equation}
\end{theorem}

For simplicity we choose to expose the proof in the case where the initial condition is the square $[-1,1]^2$.
Starting from a general convex domain is not conceptually more complicated but would involve notational complication that we wish to avoid.

Let us describe $\mathcal D(t)$ when $\mathcal D=[-1,1]^2$:
define the function $g$ on $\bbR\times \bbR_+$ by
\begin{equation}\label{defj}
g(x,t):=\begin{cases}
         \frac{x^2+t^2}{2t} \text{ if } |x|\le t,\\
                 |x|  \text{ if } |x|\ge t,
        \end{cases}
\end{equation}
let $\mathcal D_1(t)$ denote the epigraph of $ g(\cdot,t)-2$ in the base $(0,\bff_1,\bff_2)$ 
where $\bff_1:= \frac{\be_1-\be_2}{2}$, $\bff_2:= \frac{\be_1+\be_2}{2}$, 
i.e.
\begin{equation}\label{dsadsa}
 \DD_1(t):=\{ x\bff_1+y \bff_2\ | \ x\in \bbR, y \ge g(x,t)-2\},
\end{equation}

and $\DD_i$, $i=2,\dots, 4$, denote its image by rotation of an angle $(i-1)\pi/2$.
Then one has
\begin{equation}\label{dsadsa2}
\DD(t):=\bigcap_{i=1}^4 \DD_i(t).
\end{equation} 
It is a convex compact set and it is the solution of \eqref{mosdef} when
$h(\theta)=\sqrt{2}/(|\cos(\theta+\pi/4)|+|\sin(\theta+\pi/4)|)$ is 
the support function of $[-1,1]^2$.

\begin{rem}\rm \label{magfied}
 The proof of Theorem \ref{resap} also adapts quite easily to the case of positive magnetic fields $h\in(0,\infty)$. In that case
the scaling limits remains the same but time has to be rescaled by a factor $\cosh (h)/\sinh(h)$.
The non-convex case is more delicate as the limiting shape has not a nice description and it  might split into several connected components.
All of this makes the analysis more complicated on a technical point of view but we believe that an analogous result could be proved in that case with some efforts if the initial
condition is sufficiently nice, e.g.\ with smooth boundary.
\end{rem}

%
%

\subsection{Interpolating between $h=0$ and $h>0$: the weak magnetic field limit}

Remark \ref{magfied} says that the scaling limit of set of $-$ spins is somehow independent of the intensity of the magnetic fields
apart from an armless scaling factor. We want to discuss shortly here a way to obtain a non trivial intermediate regime (sometimes referred to as
a crossover regime) between the cases $h>0$
and $h=0$. We give a brief description of a conjecture concerning that case, 
based on heuristic consideration.

\medskip

The intermediate regime should take place when $h=h_L\sim \alpha L^{-1}$. When $h\gg L^{-1}$ Theorem \ref{resap} should hold whereas Theorem 
\ref{th:convex} should be valid when $h\ll  L^{-1}$.
In this intermediate regime the scaling should be $L^{2}$ and the limiting equation for the support function

\begin{equation}
  \label{eq:mix}
  \begin{cases}
\partial_t h(\theta,t)=- a(\theta)k(\theta,t)-\alpha b(\theta)\\
h(\theta,0)=h(\theta)
  \end{cases}
\end{equation}

The existence and uniqueness of solution to the above equation is a challenging issue.

\medskip

Note that this type of scaling has already be studied in the context particle system in the case where the 
interface between $+$ and $-$ is the graph of a function
(see Section \ref{prtsys}) 
and is often referred to as Weakly Asymmetric Exclusion Process or WASEP.

\medskip

It has been proved in \cite{cf:DPS, cf:Gard, cf:KOL} that in a certain sense, the height function of the particle system converges to the solution of
\begin{equation}
\partial_t u= \frac 1 2\partial^2_x u+\frac 1 2 (1-\partial_x u)^2,
\end{equation}
which is as can be checked by the reader, is equivalent to \eqref{eq:mix} in this case.
These result could be used as a starting point to get the scaling limit starting from a compact convex shape filled with $-$. Again, the major difficulty 
seem to be on the analytical side, as on the probabilistic side, most of the tools applied in the proof of Theorem \ref{th:convex} could be applied.

\subsection{Higher dimensions}

Of course, the Ising model and the dynamics are also defined in higher dimensions $d\ge 3$, for which one as also a phase transition, and they have also 
raised 
a lot of interest (in particular the case $d=3$ for obvious reason). Let us briefly talk about what should be true in that case with a focus on 
low-temperature with plus boundary condition (for high temperature, we can notice that the result of \cite{cf:LS} holds in every dimension).

\medskip

The Lifshitz conjecture for the low-temperature dynamics should hold in fact in any dimension. However rigorous results 
are even more difficult to obtain for the model
with zero-magnetic fields:
for finite $\gb$ the best known bound for the mixing time at low temperature with $+$ boundary condition is super-exponential 
(it is equal to $O(\exp(L^{d-2}(\log L)^2))$ \cite{cf:Sugi}), which is really far from the conjectured $O(L^2)$

\medskip

Concerning the zero-temperature case, results are much more precise:
in \cite{cf:CMST}, it has been proved that for $d=3$, the time necessary for the last $-$ spin to disappear starting from a full cube with $+$ boundary condition 
is of order $L^2(\log L)^{O(1)}$, a similar upper bound has been derived for arbitrary dimension $d$ in \cite{cf:L} (with no matching lower bound).
However, it seems that with actual tool, we are very far from being able to prove a shape theorem similar to Theorem \ref{th:convex} when $d=3$,
or even to get rid of the $\log$.
A first step would be to derive on a heuristic level the anisotropy function (the equivalent of $a(\theta)$ from \eqref{eq:a}), 
and it seems highly non-trivial.

\medskip

The anisotropy function should have a lot of singularity when $d\ge 3$.
For instance, consider the three-dimensional dynamics on $(\bbZ^*)^3$ starting from initial condition
\begin{equation}\begin{cases}
 \sigma(x)=-1& \text{ if } x\in (\bbZ^*)^3\cap (\bbR\times B(0,L))\\
 \sigma(x)=+1& \text{ else }, 
\end{cases}
\end{equation}
where $B(0,L)$ denote the two dimensional Euclidean ball. If $L\ge 4$, then this 
initial configuration is stable, and the system stays forever in that state:
every $-$ spin has a strict majority of agreeing neighbors and the same for $+$ spin.
However, the mean curvature is well defined and positive at every point of the interface of the cylinder.

\medskip

For the case of zero-temperature with positive $h$, it should be rather simple to show that an hyper-cube full of minus of diameter $L$ needs a time
of order $L$ to disappear. Getting the exact asymptotic and a shape Theorem seems more difficult but probably not out of reach.

\subsection{Organization of the paper}

In Section \ref{TP} we present some of the main tools of the proof. They are: general monotonicity property of the dynamics, correspondence with 
particle systems and scaling limits for interface dynamics.

\medskip

In Section \ref{convex}, we sketch in detail the proof of Theorem \ref{th:convex} from \cite{cf:LST}

\medskip

In Section \ref{asder}, we prove our main result, that is Theorem \ref{resap}, and underline the similarities and difference between this proof
and the one of the zero-magnetic fields case.

\section{Interface dynamics, correspondence with particle systems and other technical tools} \label{TP}
\subsection{Graphical construction and monotonicity}\label{grcont}

We present here a construction of the
dynamics (called sometimes the \emph{graphical construction}) that
yields nice monotonicity properties.  We consider a family of
independent Poisson clock processes $(\tau^{x})_{x\in\bar{\Z}^2}$. that is, to each site $x\in (\bbZ^*)^2$ one associates an independent random
sequence of times
$(\tau^x_{n})_{n\ge0}$, that are such that $\tau^x_{0}=0$ and
$(\tau^x_{n+1}-\tau^x_{n})_{n\ge 0}$ are IID exponential variables
with mean one.  One also defines random variables $(U_{n,x})_{n\ge 0},
x\in (\bbZ^*)^2$ that are IID Bernoulli variables,
that assume value $\pm 1$ with probability $e^{\pm h}/(2\cosh(h))$.

\medskip

Then given an initial configuration $\xi\in \{-1,1\}^{(\bbZ^*)^2}$
one constructs the dynamics $\sigma^{\xi}(t)$ starting from
$\sigma^\xi(0)=\xi$ as follows
\begin{itemize}
 \item $(\sigma_x(t))_{t\ge 0}$ is constant on the intervals of the type $[\tau^x_{n},\tau^x_{n+1})$. 
 \item $\sigma_x(\tau^x_{n})$ is chosen to be equal to $\pm1$ if a
   strict majority of the neighbors of $x$ satisfies
   $\sigma_y(\tau^x_{n})=\pm1$, and $U_{n,x}$ otherwise (this definition makes sense as, almost surely, two neighbors will not update at the same time).
\end{itemize}

This construction gives a simple way to define simultaneously the
dynamics for all initial conditions and boundary condition, using the same variable $U$ and $\tau$ (we denote by $P$ the associated
probability). 
Moreover the coupling of dynamics thus obtained preserves the natural order 
on $\{-1,+1\}^{(\bbZ^*)^2}$, given by
\begin{equation}
 \xi\ge \xi' \Leftrightarrow \xi_x\ge \xi'_x \;\;\text{for every\;\;} x \in (\bbZ^*)^2
\end{equation}
(this order is just the opposite of the inclusion order for the set of "$-$" 
spins, which is therefore also preserved).
Indeed, if $\xi\ge \xi'$ and $\sigma^\xi$ resp. $\sigma^{\xi'}$ denote the dynamics with initial condition $\sigma$ resp. $\sigma'$
using the same $\tau$ and $U$, with the above construction, one has $P$-a.s.
\begin{equation}
\forall t>0\;\;\;\;\;\;\; \sigma^{\xi}(t)\ge \sigma^{\xi'}(t).
\end{equation}
It also yields monotonicity with respect to boundary condition and other nice properties.

\subsection{The height function of the simple exclusion process}\label{prtsys}

We deal in this subsection and the next ones, of special initial conditions and boundary conditions, for which the  
the interface between $+$ and $-$ at all times is given by the graph of a function in the coordinate frame $(\bff_1, \bff_2)$ 
(recall the definition above equation \eqref{dsadsa}). This cases are easier to treat for two reasons: firstly
the rescaled interface motion can be written a functional PDE, which is easier to deal with than a flow of curves; secondly, 
there are bijective correspondence with particle systems 
that facilitates the probabilistic treatment of these interface problems (and also gives extra motivation to study them). 

\medskip

Results concerning interface models are one of  the principal building bricks for the proof of Theorem \ref{th:convex} and \ref{resap}.
We describe the correspondence in this subsection and state convergence results in the symmetric and the asymmetric cases in 
Section \ref{symasym} resp \ref{asyma}.
\medskip

We say that a set $A\subset (\bbZ^*)^2$ or $A\subset \gL_L$
is \textit{increasing} if and only if
\begin{equation}
 \forall x \in A, \left(y\ge x \Rightarrow y \in A\right),
\end{equation}
where the order in $(\bbZ^*)^2$ is the usual partial order $(y_1,y_2)\ge (x_1,x_2)$ where $y_1\ge x_1$ and $y_2\ge x_2$. 

\medskip

We define the interface between $+$ and $-$ for a configuration $\sigma$ to be the topological boundary of
\begin{equation}
 \mathcal A^-(\sigma):=\bigcup_{\{x: \ \sigma_x=-1\}} \mathcal C_x,
\end{equation}
which also is equal to the boundary of $\mathcal A^+(\sigma)$ with an analogous definition.

\medskip

If the set of $+$ spins is an increasing set for $\sigma(0)$ then it remains an increasing set for $\sigma(t)$ for all $t\ge 0$ (the only
site whose spin can flip have two $+$ neighbors and two $-$ neighbors and the reader can check that flipping one of them does 
not break the property). Also, when the set of $+$ spin is increasing, the interface between 
$+$ and $-$ is the graph of a function $\eta: \bbR\to \bbR$ in the 
frame $(0,\bff_1,\bff_2)$. The restriction of $\eta$ to $\bbZ$ is integer valued and satisfies $|\eta(x+1)-\eta(x)|=1$ for all $x\in \bbZ$
and it is linear on every interval of the type $[x,x+1]$, $x\in \bbZ$. Hence in a small abuse of notation we can also consider $\eta$
 as a function on $\bbZ$, and see the dynamics as a Markov chain on the state-space

\begin{equation}
\gO^{0} :=
\{\eta\in\bbZ^{\bbZ} \ | \ |\eta(x+1)-\eta(x)|= 1, \; \forall x \in \bbZ\}\,.
\end{equation}
In the case where the dynamic is restricted to $\gL_L$, we choose the boundary condition to be $+$ on the up and right sides and 
$-$ on the two opposite side and 
the interface function $\eta$ defined on $[-2L,2L]$ (see Figure \ref{fig:coresis2}).

 \begin{figure}[hlt]

\leavevmode
\epsfxsize =13 cm
\psfragscanon 

\psfrag{O}{\small $\bf{0}$}
\psfrag{M}{\small $M$}
\psfrag{N}{\small $N$}
\psfrag{M+N}{\small $M+N$}
\psfrag{M-N}{\small $M-N$}

\epsfbox{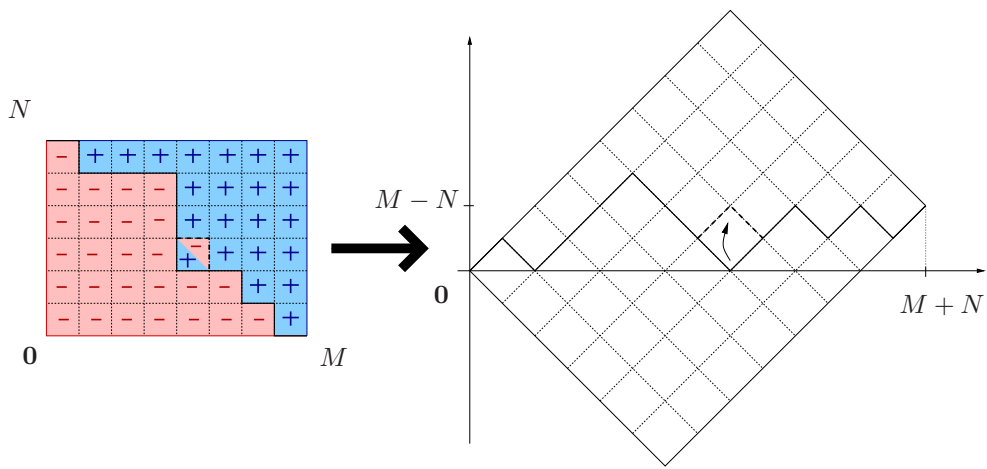}
\begin{center}
\caption{\label{fig:coresis2}
One-to-one correspondence between the dynamics in a $N\times M$ rectangle with mixed 
boundary conditions and the corner-flip dynamics on paths. This correspondence is of course also valid in the case of a $2L\times 2L$ square.
An example of possible flip update is displayed}
\end{center}
\end{figure} 

\medskip

The dynamic of the interface can be described as follows: sites $x\in \bbZ$ (or $x\in [-2L+1,2L-1]$)
are equipped with independent Poisson clocks with rate one; when a clock rings at $x$
the path $\eta$ is replaced by $\eta^{x}$ with probability $e^h/(2 \cosh(h))$ and $\eta_x$ with probability $e^h/(2 \cosh(h))$, 
where $\eta^{x}$ (and $\eta_x$) are respectively the maximum and minimum path in $\gO$ that coincides with $\eta$ on $\bbZ \setminus \{ x \}$.
The paths $\eta_x$ and $\eta^x$ differ if and only if $\eta$ has a local extremum at $x$ i.e. if $\eta(x+1)=\eta(x-1)$).

\medskip

This description is only formal in the full line case since the set of update time is dense, but is a rigorous definition 
of the generator of the chain (for a general proof of existence of a chain with such generator see \cite[Chapter 1]{cf:Lig}).

\medskip

Finally this interface dynamic can be mapped onto a one dimensional particle system.
For any $x \in \bbZ$ set $$\xi(x)=(\eta(x)-\eta(x+1)+1)/2$$ and say that a particle lies on the site $x$ when $\xi(x)=1$.
With the alternative view the dynamic can be described as follows: each particle jumps to the right with rate $e^h/(2 \cosh(h))$, 
and to the left with
rate $e^{-h}/(2 \cosh(h))$, with jumps being canceled if the aimed site is already occupied by another particle.
In the case of $\eta$ defined on a segment, we define the particle in the same manner, with the constraint that 
the extremity of the particle system are ``closed'', so that particle cannot run through them.
This system is called the simple exclusion process .

When $h=0$, the particles perform symmetric motions and one speaks about SSEP (symmetric dimple exclusion process).
When $h\in (0,\infty]$, the particles are biased towards the right and one talks about ASEP (asymmetric).
When $h=\infty$ particles can only jump to the right and the system is called TASEP (totally asymmetric)

\medskip

For all these particle systems, the only infinite volume equilibrium measures are the one under which $\xi(x)$ are IID Bernoulli variables.

\subsection{Scaling limit for the Asymmetric Simple Exclusion Process}\label{symasym}

This relation to one dimensional particle systems gave additional motivation to study these interface motion and
result for the scaling limit of the interface have been obtained. Let us cite the first one, due to Rost \cite{cf:Rost}
that concerns the case where the initial condition is $-$ in a quadrant and $+$ in the three others, in 
the totally asymmetric case (recall the definition of the function $g$ in equation \eqref{defj}).

\begin{theorem}[{\cite[Theorem 1]{cf:Rost}}]\label{rost}
Consider the stochastic Ising model on $(\Z^*)^2$ at temperature zero with $h=\infty$ and initial condition $\sigma_0$
$-$ in $\bbR^2_+$ and $+$ elsewhere.
Let $\eta(x,t)$ denote the function whose interface in $(0,\bff_1,\bff_2)$ is the interface between $+$ and 
$-$ for $\sigma(t)$ ($\eta(x,0)=|x|$). Then one has for all $\gep>0$ and $T<\infty$
\begin{equation}
\lim_{L\to \infty} \bbP\left[ \maxtwo{x\in \bbR}{t\in[0,T]} |\frac{1}{L}\eta(Lx,Lt)-g(x,t)|>\gep \right]=0.
\end{equation}
\end{theorem}

The above the result has been developed in order to be able to treat all kind of initial condition.
Let $u_0$ be a $1$-Lipshitz function and let $\sigma^L_0$ be a sequence of initial configuration for which the set of $-$ is an increasing set
and the initial interface function $\eta^L_0$ satisfies:

\begin{equation}\label{laconond}
 \max_{x\in \bbR }\left|\frac{1}{L}\eta^L_0(Lx)-u_0(x)\right|=o(1).
\end{equation}
Then define $u(x,t)$ as being the unique \textsl{viscosity solution} of 
\begin{equation}\label{viscous}\begin{cases}
 \partial_t u&= \frac{1}{2}(1-(\partial_x u)^2),\\
 u(\cdot,0)&=u_0,
\end{cases}
\end{equation}
which can be showed to be equal to (recall \eqref{defj}),
\begin{equation}\label{comendir}
 u(x,t):=\inf_{y}\{u_0(y)+ g(x-y,t)\}.
\end{equation}
We use \eqref{comendir} as the definition of $u$. It turns out that $u$ is the description of the scaling limit
of the interface when the initial condition satisfies \eqref{laconond}. Before stating the theorem, let us explain briefly 
on heuristic grounds why it is so.

\medskip

With use the particle system description given at the end of the previous section.
With this setup, the height variation $\eta(x,t)-\eta(x,0)$ is equal to twice the number of particle jumping from $x-1$ to $x$ in the 
time interval $[0,t]$.

\medskip

We consider the simplified case where the initial profile is linear with slope $s$ and the particle system is in an equilibrium configuration.
i.e.\
$(\xi_x)_{x\in \Z}$ are IID Bernoulli variables of parameter $\rho:=(1-s)/2$.
In that case the jump rate of particle from $x-1$ to $x$ is equal to 
\begin{equation*}
 \bP[\eta(x-1,t)=1\ ; \ \eta(x,t)=0]= \rho(1-\rho)= \frac{1}{4}(1-s^2).
\end{equation*}
Assuming that there is some kind of ergodicity in the system, a law of large number should hold and for any fixed $x$ a.s.\ 
\begin{equation}
 \eta(t,x)-\eta(0,x)=  \frac{1}{2}(1-s^2)t(1+o(1)).
\end{equation}
The reason why the argument also works when the original density of particle is non uniform is that 
locally, the system relaxes quickly to equilibrium so that the field $\eta(x,t)_{x\in \bbZ}$ looks locally like IID Bernoulli after a small time.

\medskip

The following Theorem was proved in \cite{cf:Rez} for ASEP in arbitrary dimension.
The reader can also refer to  \cite{cf:Sep} for a generalization to the $K$-exclusion process.

\begin{theorem}\label{visvis}
Consider the stochastic Ising model at temperature zero with $h=\infty$ and initial condition $\sigma^L_0$ as above (and $\eta^L$ the corresponding 
interface).
Then, the rescaled interface converges in law to $u$ in the following sense:
For every $\gep$, for every positive $T$ and $K$

\begin{equation}\label{viscoos}
\lim_{L\to \infty}\bP\left[\maxtwo{|x|\le K}{t\le T} |\frac 1 L \eta^L(Lx,Lt)-u(x,t)|\ge \gep\right]=0.
\end{equation}
\end{theorem}

These two results remain valid for the ASEP, but time has to be rescaled by a factor $\gamma(h)=\coth(h)$.

\subsection{Scaling limit of the Symmetric Simple Exclusion Process}\label{asyma}

The case where $h=0$ corresponds to the case where the particles perform symmetric jumps.
In that case the speed of the particle is zero, and one has to rescale time by $L^2$ to get a non-trivial scaling limit.

\medskip

It is now a classical result that any dimension, the weak-limit of the density profile of particle for the SSEP is given by the heat equation (see e.g.\ \cite[Chapter 4]{cf:KL}).
In \cite{cf:LST} we proved an analogous strong limit result for the profile $\eta$, in the case of dynamics restricted to a box:
consider the zero-temperature dynamics with zero magnetic fields on $\gL_L$ with $+$ boundary condition on the up and right sides and $-$ on the two others.
Let $t\mapsto \eta(\cdot,t)$ defined on $[-2L,2L]$ denote the function whose graph in $(0,\bff_1,\bff_2)$ is the interface between $+$ and $-$
for the zero temperature dynamics.

\medskip

Given a  $1$-Lipschitz function $v_0: [-2,2]\mapsto \bbR$ with $v_0(\pm 2)=0$, assume that one start the dynamics in 
$\gL_L$ with a sequence of initial condition 
$\sigma^L_0$ for which the interface $\eta$ satisfies 

\begin{equation}\label{initcond}
 \lim_{L\to \infty}\sup_{x\in [-2L,2L]}\left | \frac 1 L \eta(x,0)- v_0 (Lx)\right|=0
\end{equation}

Let $v$ be the solution of 
\begin{equation}\label{laplace}
\begin{cases}
 \partial_t v&= \frac 1 2 \partial^2_{x} v,\\
 v(0,t)&=v_0, \quad \forall x \in [-2,2],\\
 v(\pm 2,t)&=0, \quad \forall t\ge 0. 
\end{cases}
\end{equation}

\begin{theorem}[{\cite[Theorem 3.2]{cf:LST}}]\label{mickey}
Consider the dynamics on $\gL_L$ with the above mentionned initial condition, then
 for all $T\ge 0$ and $\gep>0$, w.h.p.
\begin{equation}
\sup_{t\in [0,T],x\in [-2,2]}\left | \frac{1}{L} \eta(Lx, L^2 t)- v(x,t) \right|\le \gep
\end{equation}
where $v$ is the solution of \eqref{laplace}.
\end{theorem}

Let us give shortly a heuristic explanation for Theorem \ref{mickey}.
The first thing is to show that the expected value of $(\eta_x)_{x\in[-L,L]}$ satisfies approximately the heat-equation.

\medskip

First notice that the expected drift of $\eta_x$ at time zero depends only on the value of $\eta_x$ and $\eta_{x\pm 1}$.
If $\eta_x$ is a local maximum it will jump down by two with rate $1/2$ whereas if it is a local minimum it will jump up by two with the same rate.
Else $\eta_x$ has drift zero so that

\begin{equation}
 \partial_t \bbE[\eta_x(t)] |_{t=0}=\begin{cases} 1 &\text{ if } \eta_{x\pm 1}(0)=\eta_{x-1}(0),\\
                                     -1 &\text{ if } \eta_{x\pm 1}(0)=\eta_{x+1}(0),\\
0 &\text{ else}.
                                    \end{cases}
\end{equation}
The reader can check that the r.h.s.\ of the above equation is equal to 
$$\frac{1}{2}(\eta_{x+1}(0)+\eta_{x-1}(0)-2\eta_x(0))=:\frac 1 2 \gD_{\mathrm d} \eta_x,$$
where $\gD_{\mathrm d}$ denotes the discrete Laplacian.
Hence, using the Markov property, one obtains that $(\bbE[\eta_x])_{x\in [-L,L]}$ satisfies

\begin{equation}
\partial_t\bbE[\eta_x(t)]=\frac 1 2 \gD_{\mathrm d}\bbE[ \eta_x(t)].
\end{equation}
Theorem \ref{mickey} is  obtained then by showing that:
\begin{itemize}
 \item The solution of the discrete heat-equation converges to the solution of the continuous one in the scaling limit 
(this is classic). 
 \item That $\eta_x(t)$ concentrates around its mean for large values of $L$. This is more delicate. We proved it by proving 
concentration for all the Fourier coefficient  in a base of eigenfunctions of $\gD_{\mathrm d}$.
\end{itemize}

Projections and trigonometry allows a heuristic derivation of \eqref{eq:meanc} from Theorem \ref{mickey} and equation \eqref{laplace}.
Indeed, it is quite reasonable to think that for the original dynamics with shrinking $-$ domain, the local drift of the interface is the same 
than for this dynamics with modified boundary condition.
However, there is a crucial ingredient missing to try to perform a proof. 

\medskip

Indeed Theorem \ref{mickey} does not say anything about the drift of the interface
around the ``poles'' of $\mal(t)$, \textit{i.e.} around the points for which one of the coordinates is extremal.
This problem was treated using another correspondence with particle system, first in \cite{cf:Spohn} where a complete sketch of proof was 
given in a special setup (periodic boundary condition).
This study was pushed further in \cite{cf:LST}  where all the technical details were handled to get a result that could be used to prove Theorem \ref{th:convex}.
We present this correspondence in the next Section.

\subsection{Dynamics near the ``poles'': Zero Range Process and Scaling Limit}\label{spoon}

Near the ''poles'', the dynamics cannot be easily reduced to an interface dynamics: indeed (see Figure \ref{fig:interditmove}) 
there are some possibilities for 
the set of $-$ to break into several connected components.
However one can introduce an auxiliary dynamics that  cancels  the transitions that 
makes $\mal(t)$ disconnected.

 \begin{figure}[hlt]

\leavevmode
\epsfxsize =10 cm
\psfragscanon 
\psfrag{pluspins}{$+$ spins}
\psfrag{minusspins}{$-$ spins}
\epsfbox{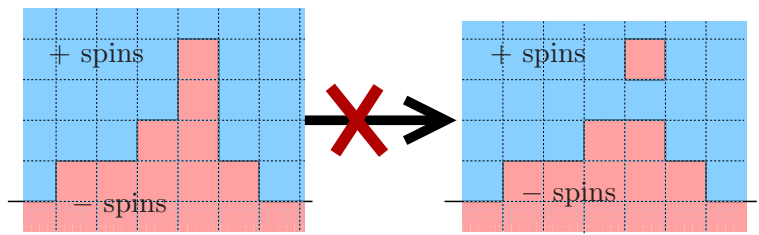}
\begin{center}
\caption{\label{fig:interditmove} An example of spin update that splits the interface into two disconnected components.
The interface dynamics presented in this section does not allow this kind of move.}
\end{center}
\end{figure}
\medskip

More precisely, we consider $\gL_L$ with boundary condition $-$ in the upper-half plane and $-$ one the lower-half plane, 
staring with an initial condition such that the interface between $+$ and $-$ is the graph of a function $[-L,L]\to \bbR$
(plus some vertical lines), and run a modified Ising dynamics that discards update if the interface after the update is not
a single connected curve. With this dynamics, the interface remains the graph of a function for all time (if one neglects the vertical lines). We call $\eta(x,t)$ the 
corresponding function (as there is no confusion possible with $\eta$ from the other section): by convention we choose it to be defined on $[-L,L]\cap \bbZ^*$ as it is piecewise constant.

\medskip

In this case also, 
we can describe the evolution of the gradient as a particle system.
For $x\in \{-L,\dots,L\}$ we set $$\xi_x(t):=\eta_{x+1}-\eta_{x-1}$$ to be the discrete gradient of $\eta$.
We say that each site  in $\{-L,\dots,L\}$ carries $|\xi_x(t)|$ particles. These particles are said to be of type $A$ if $\xi$ is positive and 
of type $B$ if $\xi$ is negative.

\medskip

Under the modified Ising dynamics depicted  above, the rules for the motion of the particles is the following:
\begin{itemize}
\item If there are $k$ particles on a site, they jump left or right with rate $1/(2k)$.
\item If a particle of type $A$ meets a particle of type $B$ they annihilate.
\end{itemize}
This kind of particle system where the jump rate depends on the number of particles on one site is called zero-range process 
(see Figure \ref{fig:partisys} for a scheme of the correspondence of 
interface dynamics with the particle system), and has been extensively studied in the literature (see e.g. \cite{cf:Andjel} where invariant measure of the 
process are studied).

 \begin{figure}[hlt]

\leavevmode
\epsfxsize =10 cm
\psfragscanon 
\epsfbox{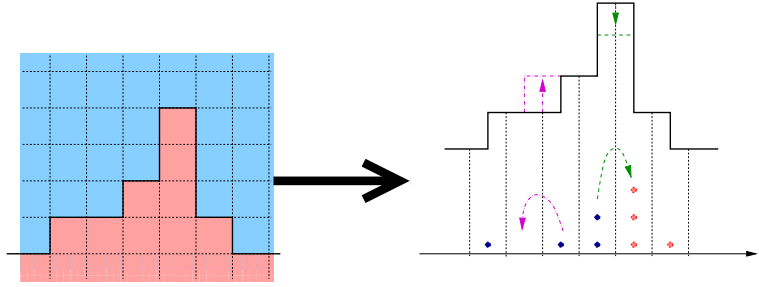}
\begin{center}
  \caption{\label{fig:partisys} Correspondence between interface
    dynamics and zero-range process.  Arrows represent possible
    motions for the interface and their representation in terms of
    particle moves.  When an A particle jumps on a B particle
    (green arrow), they annihilate.}
\end{center}
\end{figure}

\medskip

This correspondence with particle system was underlined in \cite{cf:Spohn}, and a partial proof of the scaling limit of the interface motion 
was given there: the scaling limit of $\xi$ should be the solution of the equation 

\begin{equation}\label{pajolijoli}
\partial_t w= \frac{\partial^2_x w}{1+|\partial_x w|^2}.
\end{equation}

Although we were not able to complete this proof fully (in particular we miss a statement concerning existence and regularity of the solution in \eqref{pajolijoli}), we proved a partial statement that was sufficient for the purpose 
of the proof of Theorem \ref{th:convex}.
Our statement is that, on the macroscopic scale, $\eta(x,t)$ stays close to a deterministic discrete evolution that can be though as
a discretization of \eqref{pajolijoli} \cite[Theorem 3.4]{cf:LST}. We record informally here the result for quotes in the rest of the paper

\begin{theorem}[{\cite[Theorem 3.4 and Corollary 3.5]{cf:LST}}]\label{th:spooney}
In some weak sense  equation \eqref{pajolijoli} describes the evolution of  the rescaled interface on the diffusive time-scale. 
\end{theorem}

\section{Zero magnetic fields: a detailed sketch of proof for Theorem \ref{th:convex}} \label{convex}

We expose in this section the ideas behind the proof of Theorem \ref{th:convex}. 
The first important step is to reduce the proof to an infinitesimal statement.
Using continuity properties of the conjectured scaling limit, we can show that in order to control the evolution of the $+$ 
domain for all time, it is sufficient to control the motion with a first order precision during a small time $\gep$.

\medskip

Once this is done, it suffices to iterate the statement many as much as need (order $\gep^{-1}$) to control the evolution for arbitrary positive time.
We do not develop this point further. 

\medskip

Let us state directly the two infinitesimal statement we want to prove:
the first one concerns continuity of the interface motion (which has to be used to get continuity of the motion),

\begin{proposition}
\label{trucrelouc}
Let $\DD$ be convex and with a Lipschitz curvature function.
For every $\alpha>0$, w.h.p. (recall definition \eqref{lesvoisins})
\begin{equation}\label{tipc}
\mathcal{A}_L(L^2 t) \subset L\, \mathcal{D}^{(\alpha)}\text{\;for every\;} t\ge0.
\end{equation}
Moreover, for every $\alpha>0$ there exists $\gep_1(\alpha,k_{\rm max})>0$ such that  w.h.p 
\begin{equation}\label{topc}
 \mathcal{A}_L(L^2 t) \supset L\,\mathcal{D}^{(-\alpha)} \text{\;for every\;} t\in[0,\gep_1].
\end{equation}
\end{proposition}
the second one is the control at first order of the evolution,

\begin{proposition}
\label{mainpropc}
  For all $\delta>0$ there exists $\gep_0(\delta,k_{\rm min},k_{\rm max})>0$ such that for all
  $0<\gep<\gep_0$, w.h.p.,
\begin{equation}\label{hokc}
\mathcal{A}_L(L^2 \gep) \subset L\mathcal{D}(\gep(1-\delta)),
\end{equation}
and 
\begin{equation}\label{hicc}
\mathcal{A}_L(L^2 \gep) \supset L\mathcal{D}(\gep(1+\delta))
\end{equation}
where $\DD(t)$ is the solution of $\eqref{eq:meanc}$.
\end{proposition}

The main work that remains is to prove the two inclusion bounds \eqref{hokc} and \eqref{hicc}, as Proposition \ref{trucrelouc} is more of a technical detail.
We detail  now the sketch for the proof of the upper inclusion \eqref{hokc}. The other bound is proved similarly but for technical reason it is a bit harder to expose
its proof. 

\begin{proof}[Sketch of the proof of \eqref{hokc}]

What we have to show is that after a time $\gep$ all the spins in $\mathcal D \setminus \mathcal D(\gep(1-\delta))$ (on the rescaled picture)
that were initially $-$ have turned $+$ w.h.p.\ after a time $\gep$.

\medskip

To do so we divide $\mathcal D \setminus \mathcal D(\gep(1-\delta))$ in eight regions and in each of these region, we will
try to compare our dynamics with some interface dynamics.
We consider four regions around the poles named $(A_i)_{i=1}^4$ and four others 
named $(B_i)^4_{i=1}$ (see Figure \ref{fig:zoneeight})
and one wants to show that each one of them is filled with $+$ after a time $\gep$. In what follows we will chose the $A_i$ to be very small zones 
around the poles whereas the $B_i$ will cover a proportion of the boundary of $\DD$ close to $1$.

\medskip

Due to rotational symmetries of the problem, the inclusion \eqref{hokc} reduces to proving that w.h.p.\
\begin{equation}\label{cruche}
\begin{split}
 \forall x\in B_1,   \sigma_{x}(\gep L^2)=+1,\\
 \forall x\in A_1,   \sigma_{x}(\gep L^2)=+1.
\end{split}
\end{equation}

The idea behind this division is that initially, $\gamma$ restricted to the region $A_1$ (resp $B_1$) is the graph of a function in the coordinate frame ($\be_1,\be_2$) for $A_1$ and in the frame, 
resp.\ ($\bff_1,\bff_2$). We want to modify our 
dynamics a little bit to get in the context of Theorem \ref{mickey} or Theorem \ref{th:spooney}.

 \begin{figure}[hlt]
\leavevmode
\epsfxsize =10 cm
\psfragscanon
\psfrag{A1}{$A_1$}
\psfrag{A2}{$A_2$}
\psfrag{A3}{$A_3$}
\psfrag{A4}{$A_4$}
\psfrag{B1}{$B_1$}
\psfrag{B2}{$B_2$}
\psfrag{B3}{$B_3$}
\psfrag{B4}{$B_4$}
\psfrag{P1}{$P_1$}
\psfrag{P2}{$P_2$}
\psfrag{P3}{$P_3$}
\psfrag{P4}{$P_4$} 
\psfrag{Dint}{$\DD(\gep(1-\delta))$} 
\epsfbox{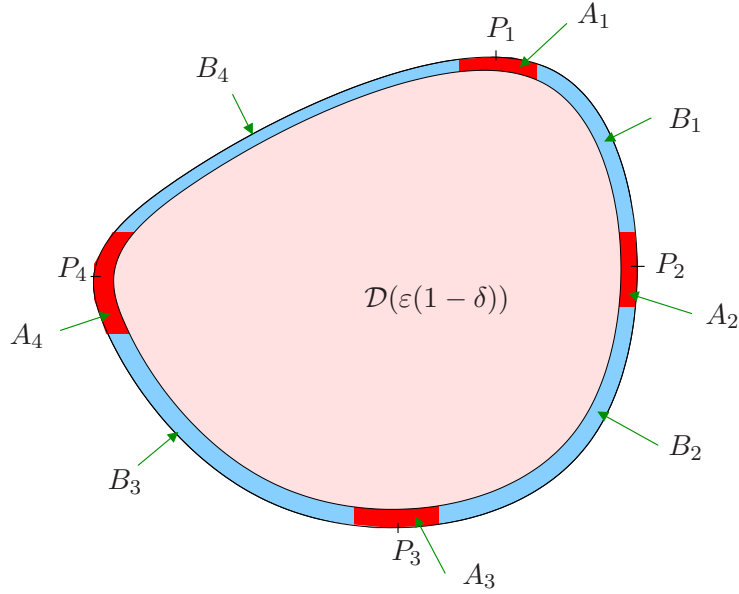}
\begin{center}
  \caption{\label{fig:zoneeight} Decomposition of $\mathcal D \setminus \mathcal D(\gep(1-\delta))$ in eight zone, 
  four small ones around the poles  $P_i$ that we call $A_i$
and four larger one away from the poles, called $B_i$}
\end{center}
\end{figure}

\medskip

The idea to prove each line of \eqref{cruche} is to replace the dynamic $\sigma$ by another one that creates more $-$ spins 
but that we can handle better.

\medskip

Let us start with the case of zone $B_1$. We look at the dynamics restricted to the rectangle of Figure \ref{fig:squarerestrict}, 
and we decide to cancel all the updates that
concerns spins of the down or right side of the rectangle, and remark that the spins on the two opposite sides remain $+$ for all time. 
The obtained dynamics $\sigma^{(1)}$ has more $-$ spins that the original one
(it can be seen from the graphical construction of Section \ref{grcont}) and falls in the setup described in Section \ref{asyma}. More precisely we have a rectangle with mixed boundary 
condition  instead of a 
square, but Theorem \ref{mickey} also 	apply in that case).

\medskip

If $\eta^{(1)}$ denotes the interface function corresponding to $\sigma^{(1)}$ (in the setup of Section \ref{asyma}) 
and $q(\cdot,t)$ denote the function whose graph in the coordinate frame ($\bff_1,\bff_2$) is the boundary of $\mathcal D(t)$ restricted to the zone $B_1$ (we chose $t$ small enough so that $q$ 
remains defined on some ''large interval'' $I$).
The first line of \eqref{cruche} is proved if one can show that 

\begin{equation}
 \frac{1}{L}\eta^{(1)}(Lx,L^2\gep)\le q(x,\gep(1+\delta)), \forall x\in I,
\end{equation}
where $I$ is an interval that is strictly included in the domain of definition $I_0$ of $q(\cdot, t)$ which sufficiently large to 
have control on the whole region $B_1$.

\medskip

Thus using Theorem \ref{mickey}, it is sufficient to prove that 

\begin{equation}\label{cocorico}
 u(x,\gep)< q(x,\gep(1+\delta)), \forall x \in I,
\end{equation}
where $u$ is the solution of equation \eqref{laplace} with initial condition $u_0=q(\cdot,0)$ corresponding to the initial position of the interface.
Continuity properties of the heat equation allows us to say that if the interval $I$ is fixed and $u_0$ is smooth enough, one has uniformly in $x\in I$,

\begin{equation}\label{dlame}
  u(x,\gep)=u_0(x)+\frac{1}{2}\partial_x^2 u_0(x)\gep+O(\gep^2).
\end{equation}
This does not hold uniformly in the full interval $I_0$ because of the Dirichlet boundary condition that makes the drift equal to zero at the extremities of the interval.
The anisotropic motion by curvature is sufficiently regular to obtain something similar for $q(x,\gep)$: uniformly on $x\in I$ 
\begin{equation}\label{dlame2}
 q(x,t)= t\frac 1 2 \partial_x^2 q(x,0)+O(t^2).
\end{equation}
At an informal level this is just a Taylor formula combined with \eqref{eq:meanc} and some trigonometry (for the projections).

\medskip

As curvature remains positive everywhere for all time and $q(x,0)=u(x,0)$, it is straight-forward
to deduce \eqref{cocorico} from \eqref{dlame} and \eqref{dlame2} for small enough $\gep$.

\begin{figure}[hlt]
\leavevmode
\epsfxsize =10 cm
\psfragscanon
\psfrag{Blockedspins}{Spins for which updates are blocked (stay $-$)}
\psfrag{Spinsthatremains+}{Spins that remains $+$ for all time} 
\epsfbox{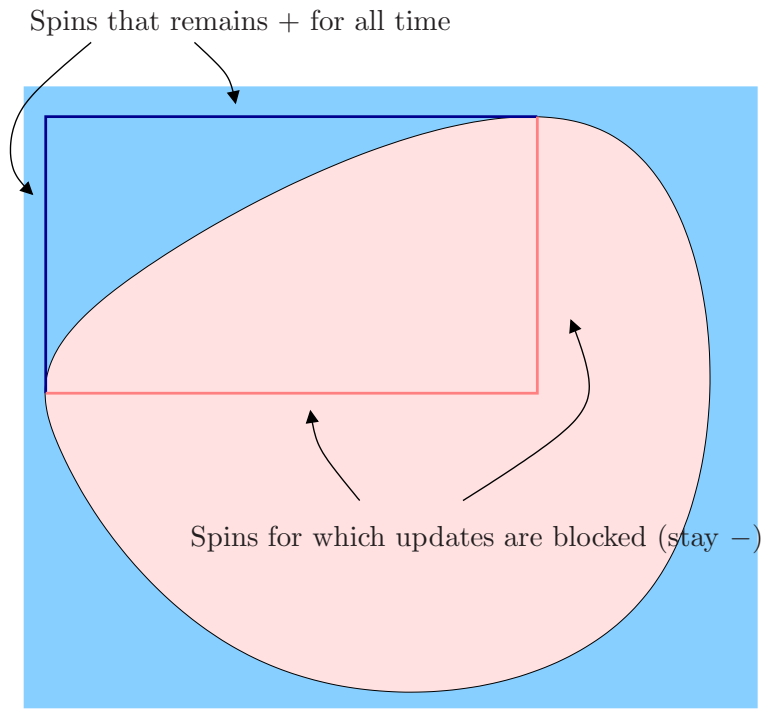}
\begin{center}
  \caption{\label{fig:squarerestrict} In order to control the dynamics on a zone of type $A$, we look at the dynamic restricted to 
a ''quadrant'' of the original shape. In order to cut dependence from what happens outside the rectangle, we block the update on a set of spins
that have to stay $-$ at all time (dark red), and we can do a monotone coupling of this dynamics with the original one. On the two opposite sides of the rectangle (dark blue), even if updates are not rejected, note that 
the spins have to stay $+$ for all time due to the majority rule.}
\end{center}
\end{figure}

\medskip

The treatment of zone a around the pole (say $A_1$) is more delicate but follows the same line. The difference is that we have 
to perform a longer chain of modification of the dynamics. Each modification makes the dynamics have more $-$ spins, or result in a dynamics that coincide with the previous one with large probability.

\medskip

We look at a dynamic restricted to a rectangle around the pole that includes $A_1$, e.g.\ take the rectangle twice as large as $A_1$
(see Figure \ref{fig:mypole}).
We want to modify the dynamics by fixing the boundary condition on the rectangle (or to freeze updates). 
The problem here is that
to get to the setup of interface dynamics of Section \ref{spoon} we need to freeze some sites with  $+$ spins as well as sites with $-$ spins so that the modified dynamics 
does not compare well with the original one.

\medskip

A way to get around this problem is to prove first that sites at a macroscopic distance (i.e.\ positive distance on the rescaled picture) from 
$\mathcal D(0)$ do not change sign ever with large probability (this is for instance the upper inclusion bound in Proposition \ref{trucrelouc}). 

\medskip

Knowing this we can
freeze spins  to $+$ in the upper-part of our rectangle provided it does not touch $\mathcal D(0)$: 
the obtained modified the dynamics will coincide with the original with large probability.

\medskip

Then if we freeze the spins in the lower-part of the rectangle to $-$ we get a dynamics that compares well with the previous one (it has more $-$ spins, see upper part of Figure
\ref{fig:mypole} upper part). 

\medskip

Once the whole boundary has been frozen, we add $-$ spins to our initial condition so that on the macroscopic scale, our initial condition is a smooth interface: 
we cancel the irregularity at the boundary (see figure \ref{fig:mypole} lower part).

\medskip

The dynamics restricted to the rectangle is not exactly an interface dynamics so we have to modify it once again. We cancel all the moves that 
break the interface between $+$ and $-$ into several connected components. With our modified initial condition (an interface that has a unique local maximum) a disconnection of the interface 
can only occur by adding a $+$ spins so the modified dynamics has one again more $-$ spins than the original one.

\medskip

Let $\sigma^{(2)}$ denote the last mentioned modified dynamics and $\eta^{(2)}(x,t)$ denote the interface function that corresponds to it. 
Let $r$ denote the function whose graph in ($\be_1,\be_2$) is the boundary of $\mathcal D(t)$ in the zone $A_1$.

We are left with proving that 

\begin{equation}\label{wister}
\frac{1}{L}\eta^{(2)}(Lx,L^2\gep)\le r(x,\gep(1+\delta)), \forall x\in J,
\end{equation}
for a small interval $J$ around the pole.

Using Theorem \ref{th:spooney}, one gets that for small $\alpha>0$, for any positive $t$ w.h.p.\ 
\begin{equation}
\frac{1}{L}\eta^{(2)}(Lx, L^2t)\le w(x,(1+\alpha)t),
\end{equation}
where $w$ is the solution of $\partial_t w=(1/2)\partial_x^2 w$ with Dirichlet boundary condition, and initial condition given by the the interface.
The number $\alpha$ can be taken arbitrarily small by reducing the size of the zone around the pole that is considered.

\medskip

With this information, proving \eqref{wister} reduces to show that
\begin{equation}\label{cpatrodur}
w(x,(1+\alpha) \gep)<  r(x,\gep(1+\delta)) \quad \forall x \in J,
\end{equation}
where $I$ is an interval corresponding to the zone $A_1$. Equation \eqref{cpatrodur} is proved in the same manner that 
we proved \eqref{cocorico}, we choose $\alpha=\delta/2$, and perform Taylor expansion at first order on both side.

\medskip

\begin{equation}\label{wescd}\begin{split}
w(x,(1+\delta/2) \gep)=w(x,0)+\gep(1+\delta/2)/2(\partial_x)^2 w(x,0)+O(\gep^2),\\
r(x,\gep(1+\delta))=r(x,0)+\gep(1+\delta)\frac{(\partial_x)^2 r(x,0)}{(1+|\partial_x r(x,0)|)^2}+O(\gep^2),
\end{split}\end{equation}
 Again the second line is formally \eqref{eq:meanc} and some trigonometry.
Then using the fact that $w(\cdot,0)=r(\cdot, 0)$ on an interval around the pole and that $|\partial_x r(x,0)|$ is uniformly small
on $J$ if $J$ is chosen small enough, we can deduce \eqref{cpatrodur} for \eqref{wescd} for $\gep$ small enough, with a small interval $J$.

 \begin{figure}[hlt]
\leavevmode
\begin{flushleft}
\epsfxsize =13 cm
\psfragscanon
\psfrag{-spins}{Spins for which updates are blocked (stay $-$)}
\psfrag{+spins}{Spins that remains $+$ for all time with large probability}
\psfrag{smoothbound}{adding $-$ spins to get a smooth boundary condition on the rescaled picture}
\epsfbox{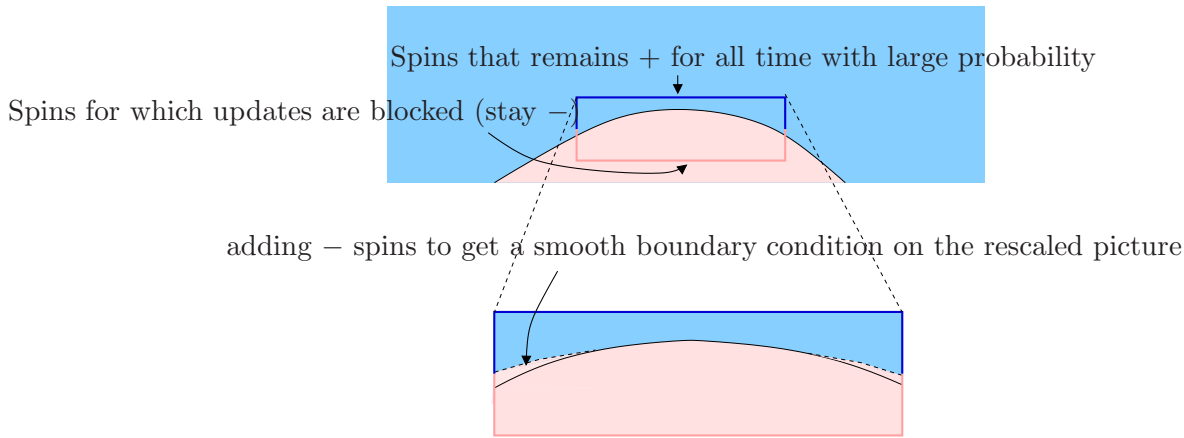}
\end{flushleft}
\begin{center}
  \caption{\label{fig:mypole} In order to control the dynamics on a zone of type $B$, we look at the dynamic restricted to 
a rectangle around the pole. In order to cut dependence from what happens outside the rectangle, we block the update on a set of spins
that have to stay $-$ at all time (dark red). On the rest of the boundary we note that 
the spins have to stay $+$ for all time with large probability so that we can couple the dynamics that reject updates on 
those site so that it coincides with the 
original dynamics with large probability.}
\end{center}
\end{figure}

Proving the other inclusion relies on the same ingredients \eqref{hicc}: first separate different zones that one wants to control, 
and then, for each zone reduce to an interface dynamics via a chain of modification.
After doing this  we use the scaling results Theorem \ref{mickey} and \ref{th:spooney} to conclude.
 Then one derives the result for the original dynamics using monotonicity.
The chain of dynamic modification is a bit longer and tedious in this case that for the upper-bound \eqref{hokc}.

\end{proof}

\section{Positive magnetic fields: proof of Theorem \ref{resap}}\label{asder}

The proof of the scaling limit of the $-$ droplet in this case uses several techniques in common with the zero-magnetic field case, but
several features of the model makes it much simpler in a sense:
\begin{itemize}
 \item The scaling limit, although being non-smooth, is a much simpler object than anisotropic motion by curvature. In particular there is no 
instantaneous travel of information on the rescaled picture: if one modifies one side of the droplet, it will take a positive time 
to have an effect on regions that are at a positive distance of where the modification took place.
 \item $+$ spins remain $+$ for all times, making the set of $-$ a decreasing set in time.
\end{itemize}

Recall that the aim of this Section is to prove Theorem \ref{resap} in the special case where $\sigma(0)$ is equal to $-$
in $[-L,L]^2\cap (\bbZ^*)^2$ and $+$ outside of it.
Recall that in that case $\DD(t)=\cap_{i=1}^4 \DD_i(t)$ where $\DD_i(t)$ is defined in equation \eqref{dsadsa}.

\medskip

The upper inclusion $\frac{1}{L}\mal(t)\subset \DD^{(\delta)}(t)$ is an easy consequence of Theorem \ref{rost}: the idea is to couple 
the dynamics with four corner growth dynamics using the same clock process for updates.
What can be rather surprising is that this simple method gives a sharp bound. 

\medskip

Indeed for the model based on simple exclusion, one can get an upper bound in
the same manner by comparing with four corner growth dynamics, 
replacing $g(x,t)$ by the solution of the heat equation with initial condition $|x|$ (and time scaling $L$ by $L^2$), but this 
would not be sharp. Indeed the shape obtained as an upper-bound with such 	a technique presents some angles on points with extremal coordinate (poles),
whereas the true scaling limit is smooth (at least $C^2$).
The reason is that the drift of the interface at the pole is positive when the pole is convex, and one needs to apply Theorem \ref{th:spooney} to take the drift 
into account.

\medskip

Here on the contrary, 
nothing happens around the poles due to the singularities of the function $b(\theta)$ (recall \eqref{beteta}) 
that is equal to zero for $\theta=i\pi/2$, $i=1,\dots,4$. 
The main thing to prove to get the lower bound $\frac{1}{L}\mal(t)\supset \DD^{(-\delta)}(t)$
is to show that the interaction between the four quadrants of our shape is quite limited.

\medskip

It is quite difficult to control directly what happens around the pole when the interface is not completely flat,
but one finds a way to bypass this problem.
Similarly to what is done in \cite{cf:Lpin} in the case of dynamics for polymer with an attractive substrate, 
one adds, in a quite artificial manner, a flat part of interfacer around the pole, and modify a bit the statement that has
to be proved (see Proposition \ref{th:bigtime} below). Adding this flat part makes the evolution of the four corners almost independent of
one another for some positive time, and thus allows to use Theorem \ref{visvis} to control the evolution of each corner and thus 
of the total shape.

\subsection{The upper-bound}

The upper-inclusion of Theorem \ref{resap} follows quite easily from Theorem \ref{rost}. We prove it in this section.
\begin{proposition}
For any $\delta>0$, for the dynamic with $h=\infty$ starting from the initial condition  $-$
in $[-L,L]^2\cap (\bbZ^*)^2$ and $+$ outside of it, 
one has w.h.p.\ 
\begin{equation}
 (\frac{1}{L}\mal(Lt))_{t\ge 0}\subset  \mathcal D^{(\delta)}(t), \forall t\ge 0,
\end{equation}
where $\mathcal D(t)$ is defined by \eqref{dsadsa}-\eqref{dsadsa2}.
\end{proposition}

\begin{proof}

Using the graphical construction of Section \ref{grcont}, we can couple the dynamics $\sigma(t)$ with initial condition $\sigma_0$: 
$-$ in $[-L,L]^2$ (we drop intersection with $(\bbZ^*)^2$ in the notation for conciseness) and $+$ elsewhere with other dynamics using 
the same clock process:
\begin{itemize}
 \item The dynamics $\sigma_1$ with initial condition $-$ in $[-L,\infty)^2$ and $+$ elsewhere.
 \item The dynamics $\sigma_2$ with initial condition $-$ in $[-L,\infty)\times[-\infty,L)$ and $+$ elsewhere.
 \item The dynamics $\sigma_3$ with initial condition $-$ in $[-\infty,L)^2$ and $+$ elsewhere.
 \item The dynamics $\sigma_4$ with initial condition $-$ in $[-\infty,L)\times[-L,\infty)$ and $+$ elsewhere.
\end{itemize}

We define $\mal^1(t)$, $\mal^2(t)$, $\mal^3(t)$, $\mal^4(t)$ analogously to $\mal(t)$ of \eqref{mal} for these four dynamics.
According to monotonicity properties of the dynamics in the initial condition one has $ \mal(t)\subset  \mal^i(t)$ for all $i\in[1,4]$ so that:
\begin{equation}\label{weirdos}
 \mal(t)\subset  \mal^1(t)\cap \mal^2(t)\cap \mal^3(t)\cap\mal^4(t).
\end{equation}
The dynamics $\sigma^1$ is the same as the one considered in
Theorem \ref{rost} with a space shift of the initial condition, and thus by Theorem \ref{rost},
\begin{equation}
 \mal^1(t)\subset \{ x\bff_1+y\bff_2 \ | \ y \ge g(x,t)-2-\delta  \}=: A^\delta_1(t).
\end{equation}
Defining $A^\delta_i(t)$, $i=2,3,4$ as rotations of $A^\delta_1(t)$ by angle $(i-1)\pi/2$, one gets analogous inclusion for $\mal^i(t)$ by symetry, 
and using \eqref{weirdos} we get that w.h.p.\

\begin{equation}
  \mal(t)\subset  A^{\delta}_1(t)\cap A^{\delta}_2(t)\cap A^{\delta}_3(t)\cap A^{\delta}_4(t)\subset \DD^{(\delta)}(t).
\end{equation}

\end{proof}

\subsection{The lower bound}

The aim of this section is to prove the other inclusion of Theorem \ref{resap}

\begin{proposition}\label{grox}
For any $\delta>0$, for the dynamic with $h=\infty$ starting from the initial condition  $-$
in $[-L,L]^2\cap (\bbZ^*)^2$ and $+$ outside of it, 
one has w.h.p.\ 
\begin{equation}
 (\frac{1}{L}\mal(Lt))_{t\ge 0}\supset  \mathcal D^{(-\delta)}(t), \quad \forall t\ge 0,
\end{equation}
where $\mathcal D(t)$ is defined by \eqref{dsadsa}-\eqref{dsadsa2}.
\end{proposition}

Let us  explain how the proof goes.
Consider  the sets  $\mal^i(t)$  defined in the previous Section.
The idea in our proof of the lower bound it to show that the inclusion used in 
\eqref{weirdos} is almost an equality.
We will decompose the proof in two steps:
\begin{itemize}
 \item First, we prove that when the rescaled-time $t$ is smaller than one the inclusion \eqref{weirdos} is indeed an
equality with large probability.
 \item Second, when the rescaled time $t$ is larger than one, we use a special strategy involving adding portions of straight-line on the interface around the pole 
in order to reduce oneself to a case where one can treat the four corner dynamics independently.  
\end{itemize}

\medskip

The control of the dynamics for times $t\le 1$ is summarized in the following lemma

\begin{lemma}\label{th:smalltime}
For any $\gd\le 0$ one has with high probability, for all $t\le 1-\delta$ one has 
\begin{equation}
 \mathcal D_t^{(-\delta)}\subset \frac{1}{L}\mal(Lt).
\end{equation}
\end{lemma}

\begin{proof}
Define $\mathcal K_L$ to be the sets of sites in $[-L,L]^2\cap(\bbZ^*)^2$ that are neighboring either the horizontal or vertical axis,

\begin{equation}
       \mathcal K_L:=\{ x\in (\bbZ^*)^2\cap [-L,L]^2 \ | \ \min( |x_1|,|x_2|)=1/2 \}.  
 \end{equation}
Set 
\begin{equation}
 \tau:=\inf\{ t\ge 0 \ | \exists x\in \mathcal K_L, \sigma_x(Lt)=+\}.
\end{equation}
The reader can check that for $t\le \tau$, the evolutions of the four corner do not interact with one another and that 
\begin{equation}\label{weirdos2}
 \mal(t)= \mal^1(t)\cap \mal^2(t)\cap \mal^3(t)\cap\mal^4(t).
\end{equation}
Moreover, as a consequence, one has
\begin{equation}
  \tau:=\inf_{i=1\dots4} \tau_i:=\inf\{ t\ge 0 \ | \  \exists x\in \mathcal K_L, \exists i\in[1,4], \sigma^i_x(Lt)=+\}.
\end{equation}
 
What remains to check is that $\tau_1$ is roughly equal to $L$.
If one uses the particle system description from Section \ref{prtsys},
the dynamic $\sigma_1$ corresponds to TASEP with {\bf step} initial condition:
i.e.\ with the negative half-line full of particles and the positive half-line empty.

\medskip

With this description,  $\tau_1$ is the first time that either
the rightmost particle, which starts from $0$, reaches $L$, or the leftmost empty space (or antiparticle)
hits $-L$. As the jump rate of the particles (and thus of antiparticles) is equal to one, the central limit Theorem for simple random walk gives us that
these time are equal to $L+O(L^{1/2})$ where the second term includes the random normal correction.
Thus one can infer that w.h.p: $$\tau_1\ge L-L^{3/4},$$ so that the same is true for $\tau$.

\medskip

From Theorem \ref{rost} one gets that with high probability for all $t\le 1$,

\begin{equation}
 \mal^1(Lt)\supset \{ x\bff_1+y\bff_2 \ | \ y \ge g(x,t)-2+\delta  \}=A^{-\delta}_1(t).
\end{equation}
so that combined with \eqref{weirdos2} one gets that w.h.p.\ for all $t\le 1-\delta$,
\begin{equation}
 \mal(Lt)\supset A^{-\delta}_1(t)\cap A^{-\delta}_2(t)\cap A^{-\delta}_3(t)\cap A^{-\delta}_4(t)\supset \mathcal D^{(-\delta)}(t).
\end{equation}

\end{proof}

We address now the issue of time larger than one. At this time, the different corners start to interact so that one has
to find a trick to regain independence. For $t\in [1,4]$, set $d(t)=2\sqrt{t}-t$.
The reader can check that for all $t\ge 1$, $\DD(t)$ is inscribed in the square $[-d(t),d(t)]^2$, i.e.\ 
\begin{equation}
 \max_{{\bf x}\in \mathcal D(t)} {\bf x}\cdot \be_1= -\min_{{\bf x}\in \mathcal D(t)} {\bf x}\cdot \be_1=d(t),
\end{equation}
the same being valid for the second coordinate.

\medskip

We will intersect $\mathcal D(t)$ with a smaller square, which gets us a shape with vertical and horizontal edges around the poles.
More precisely, given $\delta$, for $t\le 4(1-\delta)$ (we have in that case $d(t)\ge \delta$ if $\gd$ is small enough) we define the function 
\begin{equation}
 \bar g(x,t):=\begin{cases} g(x,t)-2 \text{ if } |x|\le d(t)-\delta,\\
                           |x|- (d(t)-\delta)+g(d(t)-\delta,t)-2 \text{ if } |x|\ge d(t)-\delta.
              \end{cases}
\end{equation}
It is continuous in $x$, convex in $x$ and has slope equal to $\pm 1$ outside of $[-d(t)+\delta,d(t)-\delta]$.
Define $\bar \DD_1(t)$ to be the epigraph of  $\bar g(x,t)$ in $(0,\bff_1,\bff_2)$ and $\bar \DD_i(t)$ its rotation by an angle $(i-1)\pi/2$, and set
\begin{equation}\label{bddef}
\bar \DD(t)= \bigcap_{i=1}^4 \bar \DD_i(t).
\end{equation}
The reader can check that
\begin{equation}
 \bar \DD(t)=\DD(t)\cap[-r(t),r(t)]^2,
\end{equation}
where $r(t)\in(d(t)-\delta/2,d(t))$ is the unique solution of the equations
\begin{equation}\label{defret}
  \bar g(x,t)=-x.
\end{equation}
To check that $r(t)>d(t)-\delta/2$ it is sufficient to observe that 

\begin{equation}
 \bar g(d(t)-\delta/2,t)=\delta/2+\bar g(d(t)-\delta,t)
= \delta/2-d(t)+\frac{-2 d(t)\delta+\delta^2}{2t}<d(t)-\delta/2.
\end{equation}

\begin{figure}[hlt]
\leavevmode
\epsfxsize =10 cm
\psfragscanon
\psfrag{-dt}{\small $-d(t)$}
\psfrag{-rt}{\small $-r(t)$} 
\psfrag{dt}{\small $d(t)$}
\psfrag{rt}{\small $r(t)$} 
\epsfbox{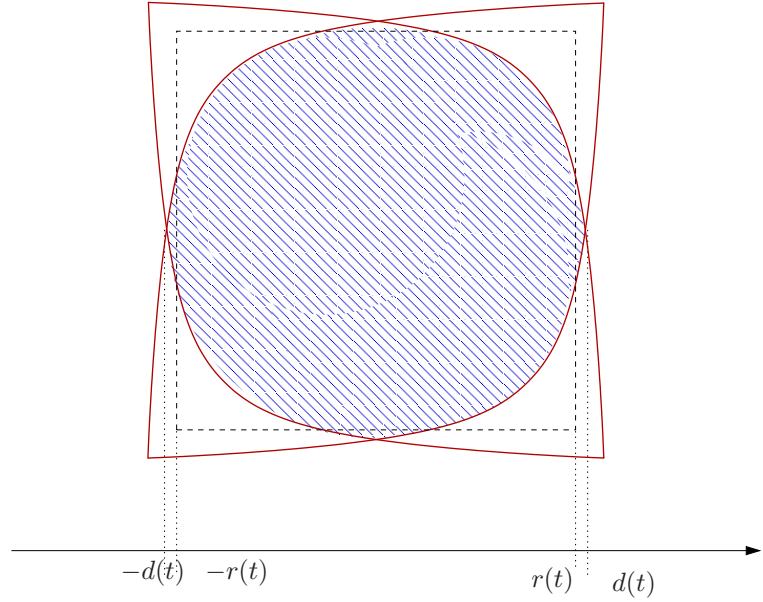}
\begin{center}
  \caption{\label{fig:lafonctiong4} The shape $\DD(t)$ (darkened on the figure) is obtained by intersection four rotated version of the epigraph of the function $x\mapsto g(x,t)-2$ 
(recall \eqref{defj}). For $t\in [1,4]$ (as above) this shape is inscribed in $[-d(t),d(t)]^2$ where $d(t)=2\sqrt{t}-t$.
We define $\bar \DD(t)$ to be the intersection of $\mathcal D(t)$ with  $[-r(t),r(t)]^2$ defined in \eqref{defret}.}
\end{center}
\end{figure}

Hence $\bar \DD(t)$ has flat parts around the pole. Thus, if one starts a dynamics from a shape approximating $\bar \DD(t)$ the four corner will not 
interact instantaneously like in the proof of Lemma \ref{th:smalltime}. We will use this fact to prove the following result

\begin{lemma}\label{th:bigtime}
For any $\gd\le 0$ and $\gep$ small enough one has for all $k\ge 0$ such that $\gep k\le 3-4\delta$  w.h.p, 
\begin{equation}\label{crgr}
\frac{1}{L}\mal( (1-\delta)(1+\gep k)L)\supset \bar \DD(1+\gep k) .
\end{equation}
As a consequence, for all $t\ge 1-\delta$ one has
\begin{equation}\label{thatsmy}
 \frac{1}{L}\mal(Lt)\supset  \DD(t)^{(-4\delta)}.
\end{equation}
\end{lemma}

\begin{proof}
The main part of the job is proving \eqref{crgr}, as \eqref{thatsmy} is deduced from it by monotonicity in $t$ of $\mal(t)$.
We proceed by induction on $k$.
For $k=0$ the result is just a consequence of Lemma \ref{th:smalltime} for $t=1-\delta$.

\medskip

Now we suppose that the result is true for $k$ and and prove it for $k+1$.
If \eqref{crgr} holds for $k$ w.h.p, then, using the graphical construction,
one can couple
$(\sigma((1-\delta)(1+\gep k)+t))_{t\ge 0}$ with $\sigma^k(t)$ a dynamic starting with initial condition
$-$ in $L \bar \DD(1+\gep k)$ and $+$ elsewhere, in such a way that
\begin{equation}
\mal((1-\delta)(1+\gep k)+t)\supset \mal^k(t), \quad \forall t \ge 0,
\end{equation}
whenever 
\begin{equation}\label{tramere}
 \sigma((1-\delta)(1+\gep k))\supset L \bar \DD(1+\gep k),
\end{equation}
 where $\mal^k$ is defined as in \eqref{mal} with  $\sigma$ replaced by $\sigma^k$.
 
 \medskip

As \eqref{tramere} holds w.h.p.\ we can prove \eqref{crgr} for $k+1$ if one proves that  
w.h.p.\
\begin{equation}\label{reptile}
 \mal^k((1-\delta)\gep)\supset L \bar \DD(1+\gep(k+1))
\end{equation}

\medskip

Now, we couple $\sigma^k(t)$ with four interface dynamics  $\sigma^{(i,k)}(t)$ starting with initial condition
 with initial condition $-$ in $L\bar \DD_i(1+\gep k)$ and $+$ elsewhere using the same clock process.

\medskip

As we remarked in the proof of Lemma \ref{th:smalltime}, one has
\begin{equation}
 \mal^k(t)\subset \bigcap_{i=1}^4 \mal^{(i,k)}(t),
\end{equation}
and this inclusion is an equality up to the first time a site near the axes changes spin.
More precisely set 
\begin{equation}
     \bar { \mathcal K}_L:=\{ x\in (\bbZ^*)^2\cap [-L r(t),Lr(t)]^2 \ | \ \min( |x_1|,|x_2|)=1/2 \}=\mathcal K_L \cap L\bar \DD(1+\gep k),  
 \end{equation}
and 
 \begin{equation}
 \bar \tau:=\inf\{ t\ge 0 \ | \ \exists x\in \bar{\mathcal K}_L, \sigma^k_x(t)=+\}.
\end{equation}
One has from the definition of our dynamics that
\begin{equation}\label{egalite}
 \mal^k(t)=\bigcap_{i=1}^4 \mal^{(i,k)}(t), \quad \forall t \le \bar \tau.
\end{equation}

\medskip

Our first task is to show that w.h.p.\ $\bar \tau\ge \gep$ in order to be allowed to use \eqref{egalite}.
We check that 
 \begin{equation}
  \bar \tau_1:= \{t\ge 0 \ |  \ \exists x\in \bar { \mathcal K}_L, \sigma^{(1,k)}_x(Lt)=-\}\ge \gep,
 \end{equation}
which by symmetry is sufficient. Indeed 

$$\bar \tau=\min_{i\in [1,4]}=\bar \tau_i$$ 

where the $\bar \tau_i$ are defined similarly to $\bar \tau_1$.

\medskip

We use again the correspondence with particle system of Section \ref{prtsys} to do that:
the dynamics $\sigma^{(1,k)}$ corresponds to a particle system with an initial condition where 
the right-most particle is located at $L(d(1+\gep k)-\delta)$ and the left-most empty-space at $-L(d(1+\gep k)-\delta)$.

\medskip

The time $\bar \tau_1$ is the first time where either the rightmost particle hits $Lr(1+\gep k)$
or the leftmost antiparticle hits $-Lr(1+\gep k)$.
Using central limit 
Theorem for the sums of exponential variables one has
w.h.p,

\begin{equation}
 \bar \tau_1\ge L((r-d)(1+\gep k)+\delta)+L^{3/4}.
\end{equation}
As $r(t)\ge d(t)-\delta/2$, one has $\bar \tau_1\ge \delta/4 L\ge \gep L$ w.h.p.\
provided that $\gep\le \delta/4$.

\medskip

Now, having shown that equation \eqref{egalite} holds for $t=\gep$ w.h.p., to get \eqref{reptile} it is sufficient to prove that 
w.h.p. (recall the definition of $\bar \DD(t)$ in \eqref{bddef}) 
\begin{equation}\label{crocodile}
  \mal^{(1,k)}((1-\delta)\gep L)\supset \bar \DD_1(1+(k+1)\gep).
\end{equation}
As the two sets are epigraphs of function in $(0,\bff_1, \bff_2)$ it is sufficient to prove an inequality between function.
Let $\eta^1(x,t)$ denote the interface function corresponding to $\sigma^{(1,k)}$.
Theorem \ref{visvis} give us the scaling limit for the evolution of $\eta^1(x,t)$, which is given by $u(x,t)$ the 
solution of \eqref{viscous} with initial condition

\begin{equation}
 u_0(x):= \bar g(x,1+\gep k).
\end{equation}
Thus \eqref{crocodile} is proved if we can show that

\begin{equation}\label{inegneg}
 u(x,(1-\delta)\gep) < \bar g(x,1+\gep(k+1)),\quad \forall x \in \bbR.
\end{equation}
The reader can readily check, using \eqref{comendir} that 
\begin{multline}\label{inefg}
  u(x,(1-\delta)\gep)=g(x,1+\gep(k+1)-\delta\gep)-2<  \bar g(x,1+\gep(k+1)),\\ \forall x\in [-d(1+\gep k)+\delta,d(1+\gep k)-\delta].
\end{multline}
What remains to show is that the inequality is also valid outside of the interval $[-d(1+\gep k)+\delta,d(1+\gep k)-\delta]$.
Set $x>d(1+\gep k)-\delta$ (by symmetry of the functions it is sufficient to check this case), as $u(\cdot,t)$ is  $1$-Lipshitz for all $t$

\begin{multline}
 u(x,(1-\delta)\gep)\le  u(d(1+\gep k)-\delta,(1-\delta)\gep)+(x-d(1+\gep k)+\delta)\\
 <\bar g(d(1+\gep k)-\delta,1+\gep(k+1))+(x-d(1+\gep k)+\delta)=\bar g(x,1+\gep(k+1)),
 \end{multline}
where the last inequality comes from the definition of  $\bar g$.
\medskip

Outside of this interval the inequality \eqref{inegneg} is also valid: indeed in that region $u(x,(1-\delta)\gep)$ is $1$-Lipshitz in $x$ and 
$\bar g(x,1+\gep(k+1))$ has slope $+ 1$ and $-1$ on the right resp.\ left of this interval \eqref{inefg}.
This concludes the proof of \eqref{crocodile} and thus of \eqref{crgr}.
\medskip

We can now detail the proof of \eqref{thatsmy}.
For $t\ge 4-4\delta$ we note that that $\DD^{(-4\delta)}(t)=\emptyset$ so that the inclusion is trivial.

\medskip
For $t\in (1-\delta, 4-4\delta]$, note that,  with high probability, \eqref{crgr} holds for all $k$ with $\gep k\le 3-4\delta$. 
Let $k_t$ be the smallest integer such that
$(1-\delta)(1+\gep k)\ge t$. As $\mal(t)$ deacreases in time, from \eqref{crgr}, for all $t \in (1-\delta, 4-4\delta]$ one has w.h.p.\
\begin{equation}\label{sicko}
 \frac{1}{L}\mal(Lt)\supset \frac{1}{L}\mal((1-\delta)(1+\gep k_t))\supset \bar \DD(1+\gep k_t).
\end{equation}
Then we note that $\bar \DD(t)\supset \DD^{(-\delta)}(t)$ and that 
\begin{equation}
 1+\gep k_t\le \frac{t}{1-\delta}+\gep,
\end{equation}
so that \eqref{sicko} implies
\begin{equation}
 \frac{1}{L}\mal(Lt)\supset \DD^{(-\delta)}\left(\frac{t}{1-\delta}+\gep\right)\supset \DD^{(-4\delta)},
\end{equation}
if $\gep$ and $\delta$ are chosen small enough.

\end{proof}

Proposition \ref{grox} is just the concatenation of Lemma \ref{th:smalltime}
with Lemma \ref{th:bigtime}.

\bigskip

{\bf Acknowledgments:} The author would like to thank the organizers of the 2012 PASI conference, 
where he had stimulating discussion with other participants, Milton Jara, for valuable bibliographic help concerning
TASEP, and François Simenhaus and Fabio Toninelli for numerous enlightening discussion on the subject. This work was partially written during the authors stay in Instituto de Matematica Pura e Applicada, 
he acknowledges the kind hospitality and the support of CNPq.

\end{document}